\documentclass[reqno,12pt]{amsart}
\usepackage{verbatim}
\usepackage{amssymb,amsmath,amsthm}
\usepackage{amsfonts}
\usepackage{array}

\def\ds{\displaystyle}

\newtheorem{theorem}{Theorem}[section]
\newtheorem{lemma}{Lemma}[section]
\newtheorem{corollary}{Corollary}[section]

\newtheorem{proposition}{Proposition}

\theoremstyle{remark}
\newtheorem{remark}[theorem]{Remark}

\usepackage{hyperref}

\newcommand{\seqnum}[1]{\href{http://oeis.org/#1}{\underline{#1}}}

\begin{document}

\title[Jacobi polynomials and congruences]{\small Jacobi polynomials and congruences involving some higher-order Catalan numbers and binomial coefficients}

\author{Kh.~Hessami Pilehrood}
\address{The Fields Institute for Research in Mathematical Sciences, 222 College St, Toronto, Ontario M5T 3J1 Canada}
\email{hessamik@gmail.com}

\author{T.~Hessami Pilehrood}
\address{The Fields Institute for Research in Mathematical Sciences, 222 College St, Toronto, Ontario M5T 3J1 Canada}
\email{hessamit@gmail.com}


\begin{abstract}
 In this paper, we  study  congruences on sums of products of binomial coefficients that can be proved  by using properties of the Jacobi polynomials. We give special attention to polynomial congruences containing Catalan numbers, second-order Catalan numbers, the sequence (\seqnum{A176898})
$S_n=\frac{{6n\choose 3n}{3n\choose 2n}}{2{2n\choose n}(2n+1)},$ and the binomial coefficients ${3n\choose n}$ and  ${4n\choose 2n}$. As an application,
we address several conjectures of Z.\ W.\ Sun on congruences of sums involving $S_n$ and we prove a cubic residuacity  criterion in terms of sums of the binomial coefficients ${3n\choose n}$ conjectured  by Z.\ H.\ Sun.
\end{abstract}

\maketitle

\section{Introduction}

In this paper, building on our previous work with Tauraso \cite{HPT}, we continue to apply properties of the Jacobi polynomials
$P_n^{(\pm 1/2, \mp 1/2)}(x)$ for proving polynomial and numerical congruences containing sums of binomial coefficients.
In particular, we derive polynomial congruences for sums involving binomial coefficients ${3n\choose n},$ ${4n\choose 2n},$
Catalan numbers (\seqnum{A000108})
$$
C_n=\frac{1}{n+1}{2n\choose n}={2n\choose n}-{2n\choose n-1}, \qquad n=0,1,2,\ldots,
$$
second-order Catalan numbers (\seqnum{A001764})
$$
C_n^{(2)}=\frac{1}{2n+1}{3n\choose n}={3n\choose n}-2{3n\choose n-1}, \qquad n=0,1,2,\ldots,
$$
and the sequence (\seqnum{A176898})
\begin{equation} \label{Sn}
S_n=\frac{{6n\choose 3n}{3n\choose n}}{2{2n\choose n}(2n+1)}, \qquad n=0,1,2,\ldots,
\end{equation}
arithmetical properties of which have been studied very recently by  Sun \cite{ZWS13} and Guo \cite{G14}.

Recall that the Jacobi polynomials $P_n^{(\alpha,\beta)}(x)$ are defined by
\begin{equation}
P_n^{(\alpha,\beta)}(x)=\frac{(\alpha+1)_n}{n!}\,F(-n,n+\alpha+\beta+1;\alpha+1;(1-x)/2),
\quad \alpha, \beta>-1,
\label{eq01}
\end{equation}
where
$$
F(a,b;c;z)=\sum_{k=0}^{\infty}\frac{(a)_k(b)_k}{k! (c)_k}z^k,
$$
is the Gauss hypergeometric function and
 $(a)_0=1,$ $(a)_k=a(a+1)\cdots(a+k-1),$ $k\ge 1,$ is the Pochhammer symbol.

The polynomials $P_n^{(\alpha,\beta)}(x)$
 satisfy the three-term recurrence relation \cite[Sect.\ 4.5]{Sz:59}
\begin{equation}
\begin{split}
2(n+1)&(n+\alpha+\beta+1)(2n+\alpha+\beta)P_{n+1}^{(\alpha,\beta)}(x)  \\
&=\left((2n+\alpha+\beta+1)(\alpha^2-\beta^2)+(2n+\alpha+\beta)_3x\right)P_n^{(\alpha,\beta)}(x)   \\
&\qquad-2(n+\alpha)(n+\beta)(2n+\alpha+\beta+2)P_{n-1}^{(\alpha,\beta)}(x)
\label{eq02}
\end{split}
\end{equation}
with the initial conditions $P_0^{(\alpha,\beta)}(x)=1,$ $P_1^{(\alpha,\beta)}(x)=(x(\alpha+\beta+2)+\alpha-\beta)/2$.

While in \cite{HPT} we studied binomial sums arising from the truncation of the series
\begin{equation} \label{eq03}
\arcsin(z)=\sum_{k=0}^{\infty}\frac{{2k\choose k}z^{2k+1}}{4^k(2k+1)}, \qquad\quad |z|\le 1,
\end{equation}
the purpose of the present paper is to consider a quadratic transformation of the Gauss hypergeometric function given by \cite[p.\ 210]{Lu}
\begin{equation} \label{eq04}
\frac{\sin(a\arcsin(z))}{a}=zF\left(\frac{1+a}{2}, \frac{1-a}{2}; \frac{3}{2}; z^2\right), \qquad |z|\le 1,
\end{equation}
which essentially can be regarded as a generalization of series (\ref{eq03}). Note that letting  $a$ approach zero in (\ref{eq04})
yields (\ref{eq03}). On the other side, identity (\ref{eq04})
serves as a source of generating functions for some special sequences of numbers including those mentioned above. Namely, for $a=1/2, 1/3, 2/3$, we have
\begin{align}
\sin\left(\frac{\arcsin(z)}{2}\right)
&=2\sum_{k=0}^{\infty}C_{2k}\left(\frac{z}{4}\right)^{2k+1}, \qquad |z|\le 1, \label{arc2} \\
\sin\left(\frac{\arcsin(z)}{3}\right)
&=\frac{z}{3}\sum_{k=0}^{\infty}C_{k}^{(2)}\left(\frac{4z^2}{27}\right)^{k}, \qquad |z|\le 1, \label{arc3} \\
\sin\left(\frac{2}{3}\arcsin(z)\right)
&=\frac{4z}{3}\sum_{k=0}^{\infty}S_{k}\left(\frac{z^2}{108}\right)^{k}, \qquad |z|\le 1.\label{gfSn}
\end{align}
In this paper, we develop a unified approach for the calculation of polynomial congruences modulo a prime $p$ arising from the truncation of the series (\ref{arc2})--(\ref{gfSn}) and polynomial congruences
involving binomial coefficients ${3k\choose k},$ ${4k\choose 2k}$
and also the sequence $(2k+1)S_k$ within various ranges of summation depending on a prime $p$.

Note that the congruences involving binomial coefficients ${3k\choose k},$ ${4k\choose 2k}$ have been studied extensively from different points of view  \cite{ZHS13, ZHS13_11, ZHS14, ZWS09, ZPS10}.
Z.\ H.\ Sun \cite{ZHS13_11, ZHS14} studied congruences for the sums
$\sum_{k=1}^{\lfloor p/3\rfloor}{3k\choose k}t^k$ and $\sum_{k=1}^{\lfloor p/4\rfloor}{4k\choose 2k}t^k$ using congruences for Lucas sequences and properties of the cubic and quartic residues.
 Sun \cite{ZHS13} also  investigated interesting connections between values of $\sum_{k=1}^{\lfloor p/3\rfloor}{3k\choose k}t^k$  (mod $p$), solubility of cubic congruences, and cubic residuacity criteria.
Zhao, Pan, and Sun \cite{ZPS10} obtained first congruences for the sums $\sum_{k=1}^{p-1}{3k\choose k}t^k$ and $\sum_{k=1}^{p-1}C_k^{(2)}t^k$ at $t=2$ with the help of some combinatorial identity. Later Z.~W.~Sun \cite{ZWS09}  gave explicit congruences for $t=-4, \frac{1}{6},\frac{1}{7}, \frac{1}{8}, \frac{1}{9}, \frac{1}{13}, \frac{3}{8}, \frac{4}{27}$ by applying  properties of third-order recurrences and cubic residues.

Our approach is based on reducing values of the finite  sums discussed above modulo a prime $p$ to values of the Jacobi polynomials $P^{(\pm 1/2, \mp 1/2)}(x),$ which is done in Section \ref{Section2}, and then investigating congruences for the Jacobi polynomials in  subsequent sections.
In Section~\ref{Section3}, we deal with polynomial congruences involving binomial coefficients ${4k\choose 2k}$ and even-indexed Catalan numbers $C_{2k}$.
In Section \ref{Section4}, we study polynomial congruences containing
binomial coefficients~${3k\choose k}$ and second-order Catalan numbers
$C_k^{(2)}$. In Sections \ref{Section5} and \ref{Section6}, we apply the theory of cubic residues developed in \cite{ZHS98} to study congruences for polynomials
 of the form
$$\sum_{k=(p+1)/2}^{\lfloor 2p/3\rfloor}{3k\choose k}t^k, \quad
\sum_{k=1}^{p-1}{3k\choose k}t^k, \quad \sum_{k=1}^{p-1}C_k^{(2)}t^k,
\quad
\sum_{k=0}^{p-1}S_kt^k, \quad\sum_{k=0}^{p-1}(2k+1)S_kt^k.
$$
As a result, we prove several cubic residuacity  criteria in terms of these sums, one of which, in terms of
$\sum_{k=(p+1)/2}^{\lfloor 2p/3\rfloor}{3k\choose k}t^k,$ confirms a question posed by Z.\ H.\ Sun  \cite[Conj.\ 2.1]{ZHS13}.

In Section \ref{Section6}, we derive polynomial congruences for the sums $\sum_{k=0}^{\lfloor p/6\rfloor}S_kt^k,$ $\sum_{k=0}^{p-1}S_kt^k,$
$\sum_{k=0}^{\lfloor p/6\rfloor}(2k+1)S_kt^k,$ $\sum_{k=0}^{p-1}(2k+1)S_kt^k$
and also give many numerical congruences which are new and have not appeared in the literature before. In particular, we show that
$$
\sum_{k=0}^{p-1}\frac{S_k}{108^k}\equiv \frac{1}{2}\left(
\frac{3}{p}\right) \pmod{p}
$$
confirming a conjecture of Z.\ W.\ Sun  \cite[Conj.\ 2]{ZWS13}. Finally, in Section \ref{Section7}, we prove a closed form formula for a companion sequence of $S_n$ answering another question of Sun~\cite[Conj.~4]{ZWS13}.

\section{Main theorem} \label{Section2}

For a non-negative integer $n,$ we consider the sequence $w_n(x)$ defined  \cite[Sect.\ 3]{HPT} by
\begin{equation}
w_n(x):=(2n+1)F(-n,n+1;3/2;(1-x)/2)=\frac{n!}{(1/2)_n} P_n^{(1/2, -1/2)}(x).
\label{eq05}
\end{equation}
  From (\ref{eq02}) it follows that $w_n(x)$
satisfies a second-order linear recurrence with constant coefficients
$$
w_{n+1}(x)=2xw_n(x)-w_{n-1}(x)
$$
and initial conditions $w_0(x)=1,$ $w_1(x)=1+2x$.
This yields the following formulae:
\begin{equation}
w_n(x)=\begin{cases}\ds
\frac{(\alpha+1)\alpha^n-(\alpha^{-1}+1)\alpha^{-n}}{\alpha-\alpha^{-1}},
&  \text{if $x \neq \pm 1$;} \\[3pt]
2n+1, &  \text{if $x=1$;} \\
(-1)^n, &   \text{if $x=-1$,}
\end{cases}
\label{eq06}
\end{equation}
where $\alpha=x+\sqrt{x^2-1}$. Note that for   $x\in (-1,1)$
we also have an
 alternative representation
\begin{equation} \label{eq23.55}
w_n(x)=\cos(n\arccos x)+\frac{x+1}{\sqrt{1-x^2}}\sin(n\arccos x).
\end{equation}
By the well-known symmetry property of the Jacobi polynomials
$$
P_n^{(\alpha, \beta)}(x)=(-1)^nP_n^{(\beta, \alpha)}(-x)
$$
and formula (\ref{eq01}), we get one more expression of $w_n(x)$ in terms of the Gauss hypergeometric function
\begin{equation} \label{eq07}
w_n(x)=(-1)^nF(-n, n+1; 1/2; (1+x)/2).
\end{equation}

For a given prime $p,$  let $D_p$ denote the set of those rational numbers whose denominator is not divisible by $p$.    Let $\varphi(m)$ be the Euler totient function and let $(\frac{a}{p})$ be the Legendre symbol. We put $(\frac{a}{p})=0$ if $p|a$.
For $c=a/b\in D_p$ written in its lowest terms, we define
$(\frac{c}{p})=(\frac{ab}{p})$ in view that the congruences
$x^2\equiv c$ (mod $p$) and $(bx)^2\equiv ab$ (mod $p$) are equivalent.
It is clear that $(\frac{c}{p})$ has all the formal properties of the ordinary Legendre symbol.
For any rational number $x,$   let $v_p(x)$ denote the $p$-adic order of $x$.
\begin{theorem} \label{TM}
Let $m$ be a positive integer with $\varphi(m)=2$, i.e., $m\in\{3, 4, 6\}$,
and let $p$ be a prime greater than $3$. Then for any $t\in D_p$,  we have
\begin{align}
\sum_{k=0}^{\lfloor p/m\rfloor}\frac{\left(\frac{1}{m}\right)_k\left(\frac{m-1}{m}\right)_k}{(2k+1)!}\,t^k&\equiv \frac{1}{1+2\lfloor p/m\rfloor}\,w_{\lfloor\frac{p}{m}\rfloor}(1-t/2) \pmod{p}, \label{eq08}\\
\sum_{k=0}^{\lfloor p/m\rfloor}\frac{\left(\frac{1}{m}\right)_k\left(\frac{m-1}{m}\right)_k}{(2k)!}\, t^k&\equiv (-1)^{\lfloor p/m\rfloor}\, w_{\lfloor\frac{p}{m}\rfloor}(t/2-1) \pmod{p}, \label{eq10}\\
\sum_{k=(p-1)/2}^{\lfloor(m-1)p/m\rfloor}\!\frac{\left(\frac{1}{m}\right)_k\left(\frac{m-1}{m}\right)_k}{(2k+1)!}\,t^k&\equiv \frac{-1}{m(1+2\lfloor p/m\rfloor)}\Bigl(w_{\lfloor\frac{(m-1)p}{m}\rfloor}(1-t/2)\!+w_{\lfloor\frac{p}{m}\rfloor}(1-t/2)\Bigr) \!\pmod{p}, \label{eq09}\\
\sum_{k=(p+1)/2}^{\lfloor(m-1)p/m\rfloor}\frac{\left(\frac{1}{m}\right)_k\left(\frac{m-1}{m}\right)_k}{(2k)!}\, t^k&\equiv \frac{(-1)^{\lfloor p/m\rfloor}}{m}\Bigl(w_{\lfloor\frac{(m-1)p}{m}\rfloor}(t/2-1)-w_{\lfloor\frac{p}{m}\rfloor}(t/2-1)\Bigr) \pmod{p}. \nonumber
\end{align}
\end{theorem}
\begin{proof}
 Let $m\in\{3, 4, 6\}$, i.e., $\varphi(m)=2$.
Suppose $p$ is an odd prime greater than $3$ and $p\equiv r$ (mod $m$), where $r\in\{1, m-1\}$. We put $n=\frac{p-r}{m}$. Then $p=mn+r$ and
 from (\ref{eq05}) we have
$$
w_n(x)=(2n+1)F\Bigl(-n, n+1; \frac{3}{2}; \frac{1-x}{2}\Bigr)=\frac{2p-2r+m}{m}\sum_{k=0}^n\frac{\left(\frac{r-p}{m}\right)_k\left(\frac{m-r+p}{m}\right)_k}{\left(\frac{3}{2}\right)_k k!}\left(\frac{1-x}{2}\right)^k.
$$
Since $(\frac{3}{2})_k=\frac{3\cdot 5\cdot\ldots\cdot(2k+1)}{2^k}$ and $2k+1\le 2n+1<p$, the denominators of the summands are coprime to $p$ and we have
$$
w_{n}(x)\!\equiv \frac{m-2r}{m}\sum_{k=0}^{\lfloor p/m\rfloor}\!\frac{\left(\frac{r}{m}\right)_k\left(\frac{m-r}{m}\right)_k}{\left(\frac{3}{2}\right)_k k!}\left(\frac{1-x}{2}\right)^k\!\!=\frac{m-2r}{m}\sum_{k=0}^{\lfloor p/m\rfloor}\!\frac{\left(\frac{1}{m}\right)_k\left(\frac{m-1}{m}\right)_k}{\left(\frac{3}{2}\right)_k k!}\left(\frac{1-x}{2}\right)^k \!\!\!\!\!\!\! \pmod{p}
$$
or
$$
w_n(x)\equiv \frac{m-2r}{m}\sum_{k=0}^{\lfloor p/m\rfloor}\frac{\left(\frac{1}{m}\right)_k\left(\frac{m-1}{m}\right)_k}{(2k+1)!}\,\left(2(1-x)\right)^k \pmod{p}.
$$
Replacing $x$ by $1-t/2$,  we get (\ref{eq08}).

Applying formula (\ref{eq07}) to $w_n(x)$, similarly as before, we get
$$
w_n(x)=(-1)^nF(-n, n+1; 1/2; (1+x)/2)=(-1)^n\sum_{k=0}^n\frac{\left(\frac{r-p}{m}\right)_k\left(\frac{p-r+m}{m}\right)_k}{\left(\frac{1}{2}\right)_k k!}\,\left(\frac{1+x}{2}\right)^k
$$
or
$$
w_n(x)\equiv (-1)^n\!\sum_{k=0}^{n}\frac{\left(\frac{1}{m}\right)_k\left(\frac{m-1}{m}\right)_k}{\left(\frac{1}{2}\right)_k k!}\left(\frac{1+x}{2}\right)^k=
(-1)^n\sum_{k=0}^n\frac{\left(\frac{1}{m}\right)_k\left(\frac{m-1}{m}\right)_k}{(2k)!}\left(2(1+x)\right)^k \pmod{p}.
$$
Substituting $t=2(1+x),$ we obtain (\ref{eq10}).

To prove the other two congruences, we consider $(m-1)p$ modulo $m$. It is clear that  $(m-1)p\equiv r$ (mod $m$), where $r\in\{1, m-1\}$. We put $n=\frac{(m-1)p-r}{m}$.  Then
$(m-1)p=mn+r$ and from (\ref{eq05}) we have
\begin{equation} \label{eq12}
w_n(x)=\frac{2(m-1)p-2r+m}{m}\sum_{k=0}^n\frac{\left(\frac{r-(m-1)p}{m}\right)_k\left(\frac{(m-1)p+m-r}{m}\right)_k}{\left(\frac{3}{2}\right)_k k!}\,
\left(\frac{1-x}{2}\right)^k.
\end{equation}
Note that $p$ divides $(\frac{3}{2})_k$ if and only if $k\ge (p-1)/2$.  Moreover, $p^2$ does not divide $(\frac{3}{2})_k$ for any $k$ from the range of summation. Similarly, we have
$$
\left(\frac{r-(m-1)p}{m}\right)_k=\prod_{l=0}^{k-1}\frac{r+ml-(m-1)p}{m}.
$$
All possible multiples of $p$ among the numbers $r+ml,$ where $0\le l\le k-1\le \frac{(m-1)p-r-m}{m}$, could be only of the form $r+ml=jp$ with $1\le j\le m-2$. This implies that $jp\equiv r\equiv -p$ (mod $m$) or $(j+1)p\equiv 0$ (mod $m$), which is impossible, since $\gcd(p,m)=1$ and $j+1<m$.
 So $p$ does not divide $(\frac{r-(m-1)p}{m})_k$.
Considering
$$
\left(\frac{(m-1)p+m-r}{m}\right)_k=\prod_{l=1}^{k}\frac{(m-1)p+ml-r}{m},
$$
we see that $p$ divides $(\frac{(m-1)p+m-r}{m})_k$ if and only if $k\ge \frac{p+r}{m}$. Moreover, $p^2$ does not divide $\bigl(\frac{(m-1)p+m-r}{4}\bigr)_k$ for any $k$
from the range of summation. Indeed, if we had $ml-r=jp$ for some $1<j\le m-1$, then $p\equiv -r\equiv jp$ (mod $m$) and therefore $p(j-1)\equiv 0$ (mod $m$),
which is impossible.
From the divisibility properties of the Pochhammer's symbols above and
(\ref{eq12}) we easily conclude that
\begin{equation*}
w_{\lfloor\frac{(m-1)p}{m}\rfloor}(x)\equiv\frac{m-2r}{m}\left(\sum_{k=0}^{\lfloor p/m\rfloor}+\!\!\sum_{k=(p-1)/2}^{\lfloor(m-1)p/m\rfloor}\right)\!
\frac{\left(\frac{r-(m-1)p}{m}\right)_k\left(\frac{(m-1)p+m-r}{m}\right)_k}{\left(\frac{3}{2}\right)_k k!}\left(\frac{1-x}{2}\right)^k \!\pmod{p}
\end{equation*}
and therefore,
\begin{equation*}
w_{\lfloor\frac{(m-1)p}{m}\rfloor}(x)\equiv \frac{m-2r}{m}\left(\sum_{k=0}^{\lfloor p/m\rfloor}+\,m\!\sum_{k=(p-1)/2}^{\lfloor(m-1)p/m\rfloor}\right)
\frac{\left(\frac{r}{m}\right)_k\left(\frac{m-r}{m}\right)_k}{\left(\frac{3}{2}\right)_k k!}\left(\frac{1-x}{2}\right)^k \pmod{p},
\end{equation*}
where for the second sum, we employed the congruence
\begin{equation*}
\begin{split}
\frac{\left(\frac{(m-1)p+m-r}{m}\right)_k}{\left(\frac{3}{2}\right)_k}&=
\frac{\frac{(m-1)p+m-r}{m}\cdot\frac{(m-1)p+2m-r}{m}\cdots\frac{(m-1)p+m\cdot\frac{p+r}{m}-r}{m}\cdots\frac{(m-1)p+mk-r}{m}}{\frac{3}{2}\cdot\frac{5}{2}\cdots
\frac{p}{2}\cdots\frac{2k+1}{2}} \\
&\equiv m\frac{\left(\frac{m-r}{m}\right)_k}{\left(\frac{3}{2}\right)_k} \pmod{p}
\end{split}
\end{equation*}
valid for $(p-1)/2\le k\le \lfloor(m-1)p/m\rfloor$.
Now by (\ref{eq08}), we obtain
$$
\frac{m}{m-2r}\,w_{\lfloor\frac{(m-1)p}{m}\rfloor}(x)\equiv \frac{1}{1+2\lfloor p/m\rfloor}\,w_{\lfloor\frac{p}{m}\rfloor}(x)+
m\!\!\!\!\!\sum_{k=(p-1)/2}^{\lfloor(m-1)p/m\rfloor}
\frac{\left(\frac{1}{m}\right)_k\left(\frac{m-1}{m}\right)_k}{(2k+1)!}\left(2(1-x)\right)^k \pmod{p}.
$$
Taking into account that
$
\left\lfloor\frac{(m-1)p}{m}\right\rfloor=p-1-\left\lfloor\frac{p}{m}\right\rfloor
$
and replacing $x$ by $1-t/2$, we get the desired congruence (\ref{eq09}).

Finally, applying formula (\ref{eq07}) and following the same line of arguments as for proving (\ref{eq09}), we have
\begin{equation*}
\begin{split}
w_{\lfloor\frac{(m-1)p}{m}\rfloor}(x)&=(-1)^{n} F(-n, n+1; 1/2; (1+x)/2) \\ &=(-1)^{\lfloor\frac{(m-1)p}{m}\rfloor}\sum_{k=0}^{\lfloor(m-1)p/m\rfloor}
\frac{\left(\frac{r-(m-1)p}{m}\right)_k\left(\frac{(m-1)p+m-r}{m}\right)_k}{\left(\frac{1}{2}\right)_k k!}\left(\frac{1+x}{2}\right)^k \\
&\equiv (-1)^{\lfloor\frac{(m-1)p}{m}\rfloor}\left(\sum_{k=0}^{\lfloor p/m\rfloor}+\sum_{k=(p+1)/2}^{\lfloor(m-1)p/m\rfloor}\right)
\frac{\left(\frac{r-(m-1)p}{m}\right)_k\left(\frac{(m-1)p+m-r}{m}\right)_k}{\left(\frac{1}{2}\right)_k k!}\left(\frac{1+x}{2}\right)^k \\
&\equiv (-1)^{\lfloor\frac{(m-1)p}{m}\rfloor}\left(\sum_{k=0}^{\lfloor p/m\rfloor}+\,m\!\sum_{k=(p+1)/2}^{\lfloor(m-1)p/m\rfloor}\right)
\frac{\left(\frac{1}{m}\right)_k\left(\frac{(m-1)}{m}\right)_k}{(2k)!}\left(2(1+x)\right)^k \pmod{p}.
\end{split}
\end{equation*}
Now by (\ref{eq09}), we obtain
$$
(-1)^{\lfloor\frac{(m-1)p}{m}\rfloor}w_{\lfloor\frac{(m-1)p}{m}\rfloor}(x)\equiv (-1)^{\lfloor\frac{p}{m}\rfloor}w_{\lfloor\frac{p}{m}\rfloor}(x)
+m\!\!\!\!\!\sum_{k=(p+1)/2}^{\lfloor(m-1)p/m\rfloor}\frac{\left(\frac{1}{m}\right)_k\left(\frac{(m-1)}{m}\right)_k}{(2k)!}\left(2(1+x)\right)^k
  \!\!\pmod{p}
$$
and after the substitution $x=t/2-1$,  we derive the last congruence of the theorem.
\end{proof}
\begin{corollary} \label{CM}
Let $m$ be a positive integer with $\varphi(m)=2,$ i.e., $m\in\{3, 4, 6\}$,
and let $p$ be a prime greater than $3$. Then for any $t\in D_p$, we have
\begin{align*}
\sum_{k=0}^{p-1}\!\frac{\left(\frac{1}{m}\right)_k\left(\frac{m-1}{m}\right)_k}{(2k+1)!}\,t^k&\!\equiv\! \frac{1}{m(1+2\lfloor p/m\rfloor)}\Bigl((m-1)w_{\lfloor\frac{p}{m}\rfloor}(1-t/2)\!-\!
w_{\lfloor\frac{(m-1)p}{m}\rfloor}(1-t/2)\Bigr) \!\!\!\pmod{p}, \\
\sum_{k=0}^{p-1}\frac{\left(\frac{1}{m}\right)_k\left(\frac{m-1}{m}\right)_k}{(2k)!}\, t^k&\equiv \frac{(-1)^{\lfloor p/m\rfloor}}{m}\Bigl((m-1)w_{\lfloor\frac{p}{m}\rfloor}(t/2-1)+
w_{\lfloor\frac{(m-1)p}{m}\rfloor}(t/2-1)\Bigr) \pmod{p}.
\end{align*}
\end{corollary}
\begin{proof}
Let $p=mn+r$, where $r\in\{1, m-1\}$. If
$n+1=\lfloor\frac{p}{m}\rfloor+1\le k\le \frac{p-3}{2}$, then $v_p((2k+1)!)=0$
and $v_p\bigl(\bigl(\frac{r}{m}\bigr)_k\bigr)\ge 1$, since the product
$\prod_{l=0}^{k-1}(r+lm)$ is divisible by $p$.

If $(m-1)n+r=\bigl\lfloor\frac{(m-1)p}{m}\bigr\rfloor+1\le k\le p-1$, then it is easy to see that $v_p((2k+1)!)=v_p((2k)!)=1$,
$v_p\bigl(\bigl(\frac{r}{m}\bigr)_k\bigr)\ge 1$,
and the product $\prod_{l=1}^{k}(lm-r)$ contains the factor $(m-1)p$.
This implies that $v_p\bigl(\bigl(\frac{m-r}{m}\bigr)_k\bigr)\ge 1$,
and therefore we have
\begin{align*}
\sum_{k=0}^{p-1}\frac{\left(\frac{1}{m}\right)_k\left(\frac{m-1}{m}\right)_k}{(2k+1)!}\,t^k&\equiv
\sum_{k=0}^{\lfloor p/m\rfloor}\frac{\left(\frac{1}{m}\right)_k\left(\frac{m-1}{m}\right)_k}{(2k+1)!}\,t^k+
\sum_{k=(p-1)/2}^{\lfloor(m-1)p/m\rfloor}\frac{\left(\frac{1}{m}\right)_k\left(\frac{m-1}{m}\right)_k}{(2k+1)!}\,t^k \pmod{p}, \\
\sum_{k=0}^{p-1}\frac{\left(\frac{1}{m}\right)_k\left(\frac{m-1}{m}\right)_k}{(2k)!}\, t^k&\equiv
 \sum_{k=0}^{\lfloor p/m\rfloor}\frac{\left(\frac{1}{m}\right)_k\left(\frac{m-1}{m}\right)_k}{(2k)!}\,t^k+
 \sum_{k=(p+1)/2}^{\lfloor(m-1)p/m\rfloor}\frac{\left(\frac{1}{m}\right)_k\left(\frac{m-1}{m}\right)_k}{(2k)!}\,t^k \pmod{p}.
\end{align*}
Finally, applying Theorem \ref{TM}, we conclude the proof of the corollary.
\end{proof}

\section{Polynomial congruences involving  Catalan numbers}
\label{Section3}

In this section, we consider applications of Theorem \ref{TM} when $m=4$. In this case, we get polynomial congruences involving even-indexed Catalan numbers $C_{2n}$ (sequence \seqnum{A048990} in the OEIS \cite{OEIS}) and binomial
coefficients~${4n\choose 2n}$  (sequence \seqnum{A001448}).
\begin{theorem} \label{T1}
Let $p$ be an odd prime and let $t\in D_p$. Then
\begin{align*}
\sum_{k=0}^{\lfloor p/4\rfloor}C_{2k}t^k&\equiv 2(-1)^{\frac{p-1}{2}}w_{\lfloor\frac{p}{4}\rfloor}(1-32t) \pmod{p}, \\
\sum_{k=(p-1)/2}^{\lfloor3p/4\rfloor}C_{2k}t^k&\equiv \frac{(-1)^{\frac{p+1}{2}}}{2}\Bigl(w_{\lfloor\frac{3p}{4}\rfloor}(1-32t)+w_{\lfloor\frac{p}{4}\rfloor}(1-32t)\Bigr) \pmod{p}, \\
\sum_{k=0}^{\lfloor p/4\rfloor}{4k\choose 2k} t^k&\equiv \left(\frac{-2}{p}\right) w_{\lfloor\frac{p}{4}\rfloor}(32t-1) \pmod{p}, \\
\sum_{k=(p+1)/2}^{\lfloor3p/4\rfloor}{4k\choose 2k} t^k&\equiv \frac{1}{4}\left(\frac{-2}{p}\right)\Bigl(w_{\lfloor\frac{3p}{4}\rfloor}(32t-1)-w_{\lfloor\frac{p}{4}\rfloor}(32t-1)\Bigr) \pmod{p}.
\end{align*}
\end{theorem}
\begin{proof}
We put $m=4$ in Theorem \ref{TM}. Then for any odd prime $p$, we have $p=4l+r$, where $l$ is non-negative integer and $r\in\{1,3\}$.
Hence,
$$
\frac{1}{1+2\lfloor p/4\rfloor}=\frac{1}{1+2l}=\frac{2}{p+2-r}\equiv \frac{2}{2-r}=2(-1)^{(p-1)/2} \pmod{p}.
$$
Moreover, $(-1)^{\lfloor p/4\rfloor}=(-1)^l=(\frac{-2}{p})$. Now noticing that
$$
C_{2k}=\frac{1}{2k+1}{4k\choose 2k}=\frac{(4k)!}{(2k)!(2k+1)!}=\frac{(\frac{3}{4})_k(\frac{1}{4})_k}{(2k+1)!}\,(64)^k
$$
and replacing $t$ by $64t$ in Theorem \ref{TM}, we get the desired congruences.
\end{proof}
\begin{corollary} \label{C1}
Let $p$ be an odd prime and let $t\in D_p$. Then
\begin{align*}
\sum_{k=0}^{p-1}C_{2k}t^k&\equiv \frac{1}{2}\left(\frac{-1}{p}\right)\Bigl(3w_{\lfloor\frac{p}{4}\rfloor}(1-32t)-w_{\lfloor\frac{3p}{4}\rfloor}(1-32t)\Bigr) \pmod{p}, \\
\sum_{k=0}^{p-1}{4k\choose 2k} t^k&\equiv \frac{1}{4}\left(\frac{-2}{p}\right)\Bigl(w_{\lfloor\frac{3p}{4}\rfloor}(32t-1)+3w_{\lfloor\frac{p}{4}\rfloor}(32t-1)\Bigr) \pmod{p}.
\end{align*}
\end{corollary}
Evaluating values of the sequences $w_{\lfloor\frac{p}{4}\rfloor}(x)$ and $w_{\lfloor\frac{3p}{4}\rfloor}(x)$ modulo $p$, we get numerical congruences for the above sums. Here are some typical examples.
\begin{corollary}
 Let $p$ be a prime greater than $3$. Then
\begin{align*}
\sum_{k=0}^{p-1}\frac{C_{2k}}{16^k}&\equiv
\left(\frac{2}{p}\right),
\qquad\quad
\sum_{k=0}^{\lfloor p/4\rfloor}\frac{C_{2k}}{16^k}\equiv
2\left(\frac{2}{p}\right) \pmod{p}, \\
\sum_{k=0}^{p-1}\frac{{4k\choose 2k}}{16^k}&\equiv \frac{1}{4}\left(\frac{2}{p}\right),
\qquad\,
\sum_{k=\frac{p+1}{2}}^{\lfloor3p/4\rfloor}\frac{{4k\choose 2k}}{16^k}\equiv -\frac{1}{4}\left(\frac{2}{p}\right) \pmod{p},
\end{align*}
\vspace{-0.5cm}
\begin{align*}
\sum_{k=0}^{\lfloor p/4\rfloor}\frac{C_{2k}}{32^k}&\equiv
2\left(\frac{-1}{p}\right)(-1)^{\lfloor p/8\rfloor},
\qquad
\sum_{k=0}^{\lfloor p/4\rfloor}\frac{{4k\choose 2k}}{32^k}\equiv \left(\frac{-2}{p}\right)(-1)^{\lfloor p/8\rfloor} \pmod{p}, \\[2pt]
\sum_{k=0}^{p-1}\frac{C_{2k}}{32^k}&\equiv
\begin{cases}
\ds
(-1)^{(p-1)/2+\lfloor p/8\rfloor} \pmod{p},
&  \text{if $p\equiv \pm 1 \pmod{8};$} \\[3pt]
\ds 2(-1)^{(p-1)/2+\lfloor p/8\rfloor} \pmod{p}, &   \text{if $p\equiv \pm 3  \pmod{8},$}
\end{cases}
\\[3pt]
\sum_{k=0}^{p-1}\frac{{4k\choose 2k}}{32^k}&\equiv
\begin{cases}
\ds
\Bigl(\frac{-2}{p}\Bigr)(-1)^{\lfloor p/8\rfloor} \pmod{p},
&  \text{if $p\equiv \pm 1 \pmod{8}$;} \\[3pt]
\ds \frac{1}{2}\Bigl(\frac{-2}{p}\Bigr)(-1)^{\lfloor p/8\rfloor} \pmod{p}, &   \text{if $p\equiv \pm 3 \pmod{8}$.}
\end{cases}
\end{align*}
\end{corollary}
\begin{proof}
The proof easily follows from the fact that
\begin{equation} \label{w-101}
w_n(-1)=(-1)^n, \quad w_n(0)=(-1)^{\lfloor n/2\rfloor}, \quad\text{and}\quad w_n(1)=2n+1.
\end{equation}
\end{proof}
\begin{corollary}
Let $p$ be a prime greater than $3$. Then
\begin{align*}
\left(\frac{2}{p}\right)\sum_{k=0}^{p-1}\frac{C_{2k}}{64^k}&\equiv
\begin{cases}
\ds
1 \pmod{p},
&  \text{if $p\equiv \pm 1 \pmod{12}$;} \\
\ds -7/2
 \pmod{p}, &   \text{if $p\equiv \pm 5  \pmod{12}$,}
\end{cases}
\\[2pt]
\left(\frac{2}{p}\right)\sum_{k=0}^{p-1}
\frac{{4k\choose 2k}}{64^k}&\equiv
\begin{cases}
\ds
1 \pmod{p},
& \text{if $p\equiv \pm 1 \pmod{12};$} \\
\ds 1/4 \pmod{p}, &   \text{if $p\equiv \pm 5 \pmod{12}$,}
\end{cases}
\\[2pt]
\sum_{k=0}^{p-1}C_{2k}\left(\frac{3}{64}\right)^k&\equiv
\begin{cases}
\ds
1 \pmod{p},
&  \text{if $p\equiv \pm 1 \pmod{12}$;} \\
\ds 1/2
 \pmod{p}, &   \text{if $ p\equiv \pm 5  \pmod{12}$,}
\end{cases}
\\[2pt]
\sum_{k=0}^{p-1}
{4k\choose 2k}\left(\frac{3}{64}\right)^k&\equiv
\begin{cases}
\ds
1 \pmod{p},
&  \text{if $p\equiv \pm 1 \pmod{12}$;} \\
\ds -5/4 \pmod{p}, &   \text{if $p\equiv \pm 5 \pmod{12}$.}
\end{cases}
\end{align*}
\end{corollary}
\begin{proof}
 We can easily evaluate by (\ref{eq23.55}),
\begin{equation} \label{12-1}
w_n(1/2)=2\cos\Bigl(\frac{\pi(n-1)}{3}\Bigr)\quad\text{and} \quad w_n(-1/2)=\frac{2}{\sqrt{3}}\sin\Bigl(\frac{\pi(2n+1)}{3}\Bigr).
\end{equation}
Hence we  obtain
\begin{align*}
w_{\lfloor\frac{p}{4}\rfloor}(1/2)&\equiv
\begin{cases}
\ds
(-1)^{\lfloor p/4\rfloor} \pmod{p},
&  \text{if $p\equiv \pm 1 \pmod{12}$;} \\
\ds -2(-1)^{\lfloor p/4\rfloor} \pmod{p}, &   \text{if $p\equiv \pm 5 \pmod{12}$,}
\end{cases}
\\[2pt]
w_{\lfloor\frac{p}{4}\rfloor}(-1/2)&\equiv
\begin{cases}
\ds
(-1)^{(p-1)/2} \pmod{p},
&  \text{if $p\equiv \pm 1 \pmod{12}$;} \\
\ds 0 \pmod{p}, &   \text{if $p\equiv \pm 5 \pmod{12}$}
\end{cases}
\end{align*}
and $w_{\lfloor 3p/4\rfloor}(1/2)\equiv (-1)^{\lfloor p/4\rfloor}$,
$w_{\lfloor 3p/4\rfloor}(-1/2)\equiv (-1)^{(p-1)/2}$ (mod $p$).
Applying Corollary \ref{C1} and the equality
$(-1)^{(p-1)/2+\lfloor p/4\rfloor}=\bigl(\frac{2}{p}\bigr)$,
we get the desired congruences.
\end{proof}
\begin{lemma} \label{L11}
For any $x\ne \pm 1$, we have
$$
w_n(2x^2-1)=\frac{\alpha^{2n+1}-\alpha^{-2n-1}}{\alpha-\alpha^{-1}},
\qquad\text{where}\quad \alpha=x+\sqrt{x^2-1}.
$$
\end{lemma}
\begin{proof}
By (\ref{eq06}), we obtain
\begin{equation*}
\begin{split}
w_{2n}(x)&=
\frac{(\alpha+1)\alpha^{2n}
-(\alpha^{-1}+1)\alpha^{-2n}}{\alpha-\alpha^{-1}} \\[2pt]
&=\frac{(\alpha^2+1)\alpha^{2n}-(\alpha^{-2}+1)\alpha^{-2n}+(\alpha+\alpha^{-1})(\alpha^{2n}-\alpha^{-2n})}{\alpha^2-\alpha^{-2}}\\
&=w_n(2x^2-1)+\frac{\alpha^{2n}-\alpha^{-2n}}{\alpha-\alpha^{-1}}.
\end{split}
\end{equation*}
This implies
\begin{equation*}
w_n(2x^2-1)=w_{2n}(x)-
\frac{\alpha^{2n}-\alpha^{-2n}}{\alpha-\alpha^{-1}}=
\frac{\alpha^{2n+1}-\alpha^{-2n-1}}{\alpha-\alpha^{-1}},
\end{equation*}
and the lemma follows.
\end{proof}
\begin{lemma} \label{L12}
Let $p$ be a prime, $p>3$, and let $x\in D_p$.
Then
\begin{align*}
w_{\lfloor\frac{p}{4}\rfloor}(2x^2-1)&\equiv \frac{1}{2}\left(\frac{-2}{p}\right)
\left(\left(\frac{1-x}{p}\right)+\left(\frac{1+x}{p}\right)\right)
\pmod{p}, \\
w_{\lfloor\frac{3p}{4}\rfloor}(2x^2-1)&\equiv \frac{1}{2}\left(\frac{-2}{p}\right)
\left(\left(\frac{1-x}{p}\right)(1+2x)+
\left(\frac{1+x}{p}\right)(1-2x)\right)
\pmod{p}.
\end{align*}
\end{lemma}
\begin{proof} First, we suppose that $x\ne \pm 1$. Then, by Lemma \ref{L11},
if $p\equiv 1$ (mod $4$),  we have
\begin{equation*}
\begin{split}
w_{\lfloor\frac{p}{4}\rfloor}(2x^2-1)&=w_{\frac{p-1}{4}}(2x^2-1)=
\frac{\alpha^{\frac{p+1}{2}}
-\alpha^{-\frac{p+1}{2}}}{\alpha-\alpha^{-1}} \\
&=
\frac{\left(\sqrt{\frac{x+1}{2}}
+\sqrt{\frac{x-1}{2}}\right)^{p+1}-
\left(\sqrt{\frac{x+1}{2}}
-\sqrt{\frac{x-1}{2}}\right)^{p+1}}{2\sqrt{x^2-1}} \\
&=\frac{1}{\sqrt{x^2-1}}
\underset{k\, is\, odd}{\sum\limits_{k=1}^{p}}{p\choose k}
\left(\frac{x-1}{2}\right)^{\frac{k}{2}}
\left(\frac{x+1}{2}\right)^{\frac{p+1-k}{2}} \\
&\,\,\, +\frac{1}{\sqrt{x^2-1}}
\underset{k\, is\, even}{\sum\limits_{k=0}^{p-1}}{p\choose k}
\left(\frac{x-1}{2}\right)^{\frac{k+1}{2}}
\left(\frac{x+1}{2}\right)^{\frac{p-k}{2}} \\
&\equiv \frac{(x-1)^{\frac{p-1}{2}}}{2^{\frac{p+1}{2}}}+
\frac{(x+1)^{\frac{p-1}{2}}}{2^{\frac{p+1}{2}}}
\equiv \frac{1}{2}\left(\frac{2}{p}\right)
\left(\left(\frac{x-1}{p}\right)+\left(\frac{x+1}{p}\right)\right)
\pmod{p}.
\end{split}
\end{equation*}
If $p\equiv 3$ (mod $4$), then, by Lemma \ref{L11}, we have
\begin{equation*}
\begin{split}
w_{\lfloor\frac{p}{4}\rfloor}(2x^2&-1)=w_{\frac{p-3}{4}}(2x^2-1)=
\frac{\alpha^{\frac{p-1}{2}}
-\alpha^{-\frac{p-1}{2}}}{\alpha-\alpha^{-1}} \\
&=
\frac{\left(\sqrt{\frac{x+1}{2}}
+\sqrt{\frac{x-1}{2}}\right)^{p-1}-
\left(\sqrt{\frac{x+1}{2}}
-\sqrt{\frac{x-1}{2}}\right)^{p-1}}{2\sqrt{x^2-1}} \\
&=\frac{\left(\sqrt{\frac{x+1}{2}}
-\sqrt{\frac{x-1}{2}}\right)\left(\sqrt{\frac{x+1}{2}}
+\sqrt{\frac{x-1}{2}}\right)^{p}-
\left(\sqrt{\frac{x+1}{2}}
+\sqrt{\frac{x-1}{2}}\right)\left(\sqrt{\frac{x+1}{2}}
-\sqrt{\frac{x-1}{2}}\right)^{p}}{2\sqrt{x^2-1}}.
\end{split}
\end{equation*}
Simplifying as in the previous case, we get modulo $p$,
$$
w_{\lfloor\frac{p}{4}\rfloor}(2x^2-1)\equiv \frac{1}{2}\left(
\left(\frac{x-1}{2}\right)^{\frac{p-1}{2}}-
\left(\frac{x+1}{2}\right)^{\frac{p-1}{2}}\right)
\equiv \frac{1}{2}\left(\frac{2}{p}\right)
\left(\left(\frac{x-1}{p}\right)-\left(\frac{x+1}{p}\right)\right),
$$
and the first congruence of the lemma follows.
Similarly, to prove the second congruence, we consider two cases.
If $p\equiv 1$ (mod $4$), then we get
\begin{equation*}
\begin{split}
w_{\lfloor\frac{3p}{4}\rfloor}(2x^2&-1)=w_{\frac{3(p-1)}{4}}(2x^2-1)=
\frac{\alpha^{\frac{3p-1}{2}}
-\alpha^{-\frac{3p-1}{2}}}{\alpha-\alpha^{-1}} \\
&=
\frac{\left(\sqrt{\frac{x+1}{2}}
+\sqrt{\frac{x-1}{2}}\right)^{3p-1}-
\left(\sqrt{\frac{x+1}{2}}
-\sqrt{\frac{x-1}{2}}\right)^{3p-1}}{2\sqrt{x^2-1}} \\
&=\frac{\left(\sqrt{\frac{x+1}{2}}
-\sqrt{\frac{x-1}{2}}\right)\left(\sqrt{\frac{x+1}{2}}
+\sqrt{\frac{x-1}{2}}\right)^{3p}-
\left(\sqrt{\frac{x+1}{2}}
+\sqrt{\frac{x-1}{2}}\right)\left(\sqrt{\frac{x+1}{2}}
-\sqrt{\frac{x-1}{2}}\right)^{3p}}{2\sqrt{x^2-1}}.
\end{split}
\end{equation*}
Simplifying the right-hand side modulo $p$, we obtain
\begin{equation*}
 \frac{1}{2}
\left(\frac{x-1}{2}\right)^{\frac{3p-1}{2}}-
\frac{1}{2}\left(\frac{x+1}{2}\right)^{\frac{3p-1}{2}}+
\frac{3}{2}\left(\frac{x-1}{2}\right)^{\frac{p-1}{2}}
\left(\frac{x+1}{2}\right)^{p}-
\frac{3}{2}\left(\frac{x+1}{2}\right)^{\frac{p-1}{2}}
\left(\frac{x-1}{2}\right)^{p}
\end{equation*}
and therefore,
\begin{equation*}
w_{\lfloor\frac{3p}{4}\rfloor}(2x^2-1)\equiv
\frac{1}{2}\left(\frac{2}{p}\right)
\left(\left(\frac{x-1}{p}\right)(1+2x)+
\left(\frac{x+1}{p}\right)(1-2x)
\right)
\pmod{p}.
\end{equation*}
If $p\equiv 3$ (mod $4$), then
\begin{equation*}
\begin{split}
w_{\lfloor\frac{3p}{4}\rfloor}(2x^2-1)&=w_{\frac{3p-1}{4}}(2x^2-1)=
\frac{\alpha^{\frac{3p+1}{2}}
-\alpha^{-\frac{3p+1}{2}}}{\alpha-\alpha^{-1}} \\
&=
\frac{\left(\sqrt{\frac{x+1}{2}}
+\sqrt{\frac{x-1}{2}}\right)^{3p+1}-
\left(\sqrt{\frac{x+1}{2}}
-\sqrt{\frac{x-1}{2}}\right)^{3p+1}}{2\sqrt{x^2-1}} .
\end{split}
\end{equation*}
Simplifying the right-hand side modulo $p$, we get
\begin{equation*}
 \frac{1}{2}
\left(\frac{x-1}{2}\right)^{\frac{3p-1}{2}}+
\frac{1}{2}\left(\frac{x+1}{2}\right)^{\frac{3p-1}{2}}+
\frac{3}{2}\left(\frac{x-1}{2}\right)^{\frac{p-1}{2}}
\left(\frac{x+1}{2}\right)^{p}+
\frac{3}{2}\left(\frac{x+1}{2}\right)^{\frac{p-1}{2}}
\left(\frac{x-1}{2}\right)^{p}
\end{equation*}
and therefore,
\begin{equation*}
w_{\lfloor\frac{3p}{4}\rfloor}(2x^2-1)\equiv
\frac{1}{2}\left(\frac{2}{p}\right)
\left(\left(\frac{x-1}{p}\right)(2x+1)+
\left(\frac{x+1}{p}\right)(2x-1)
\right)
\pmod{p},
\end{equation*}
as required. If $x=\pm 1$, then, by (\ref{eq06}), we have $w_{\lfloor\frac{p}{4}\rfloor}(1)=2\lfloor\frac{p}{4}\rfloor+1\equiv (-1)^{(p-1)/2}/2 \pmod{p}$
and $w_{\lfloor\frac{3p}{4}\rfloor}(1)=2\lfloor\frac{3p}{4}\rfloor+1\equiv (-1)^{(p+1)/2}/2$ (mod $p$), which completes the proof of the lemma.
\end{proof}
  From Lemma \ref{L12} and Corollary \ref{C1} we immediately deduce the following result.
\begin{theorem}
Let $p$ be a prime, $p>3$, and let $t\in D_p$. Then
\begin{align*}
\sum_{k=0}^{p-1}C_{2k}\left(\frac{1-t^2}{16}\right)^k&\equiv
\frac{1}{2}\left(\frac{2}{p}\right)\left(\left(\frac{1-t}{p}\right)(1-t)+\left(\frac{1+t}{p}\right)(1+t)\right) \pmod{p}, \\
\sum_{k=0}^{p-1}{4k\choose 2k}t^{2k}&\equiv\frac{1}{2}
\left(\left(\frac{1-4t}{p}\right)(1+2t)+\left(\frac{1+4t}{p}\right)(1-2t)\right) \pmod{p}.
\end{align*}
\end{theorem}
\begin{proof}
  From Corollary \ref{C1} we have
$$
\sum_{k=0}^{p-1}C_{2k}\left(\frac{1-t^2}{16}\right)^k\equiv \frac{1}{2}\left(\frac{-1}{p}\right)\bigl(3w_{\lfloor\frac{p}{4}\rfloor}(2t^2-1)-
w_{\lfloor\frac{3p}{4}\rfloor}(2t^2-1)\bigr) \pmod{p}
$$
and
$$
\sum_{k=0}^{p-1}{4k\choose 2k}t^{2k}\equiv \frac{1}{4}\left(\frac{-2}{p}\right)\bigl(3w_{\lfloor\frac{p}{4}\rfloor}(32t^2-1)+
w_{\lfloor\frac{3p}{4}\rfloor}(32t^2-1)\bigr) \pmod{p}.
$$
Now by Lemma \ref{L12} with $x$ replaced by $4t$ for the last congruence, we conclude the proof.
\end{proof}
\begin{theorem}
Let $p$ be a prime, $p>3$, and let $a, b\in {\mathbb Z}$, $ab\not\equiv 0 \pmod{p}$, and $a\not\equiv b \pmod{p}$. Then we have the following congruences modulo $p$:
\begin{align*}
\sum_{k=0}^{p-1}C_{2k}\,\frac{(a-b)^{2k}}{(-64ab)^k}&\equiv
\begin{cases}
\ds
\frac{(ab)^{\frac{p-1}{4}}}{2(a-b)}
\left((3a+b)\left(\frac{b}{p}\right)-(3b+a)\left(\frac{a}{p}\right)
\right),
&  \text{if $p\equiv  1 \pmod{4}$;} \\[5pt]
\ds \frac{(ab)^{\frac{p+1}{4}}}{2(b-a)}\left(\frac{3a+b}{a}
\left(\frac{b}{p}\right)-\frac{3b+a}{b}\left(\frac{a}{p}\right)\right),
  &  \text{if $p\equiv 3  \pmod{4}$,}
\end{cases}
\\[2pt]
\sum_{k=0}^{p-1}
{4k\choose 2k}\frac{(a+b)^{2k}}{(64ab)^k}&\equiv
\begin{cases}
\ds
\frac{\bigl(\frac{2}{p}\bigr)(ab)^{\frac{p-1}{4}}}{4(a-b)}\left(
(3a-b)\left(\frac{b}{p}\right)-(3b-a)\left(\frac{a}{p}\right)\right),
&  \text{if $p\equiv  1 \pmod{4}$;} \\[5pt]
\ds \frac{\bigl(\frac{2}{p}\bigr)(ab)^{\frac{p+1}{4}}}{4(b-a)}\left(
\frac{3a-b}{a}\left(\frac{b}{p}\right)-\frac{3b-a}{b}\left(
\frac{a}{p}\right)\right),  &   \text{if $p\equiv 3 \pmod{4}$.}
\end{cases}
\end{align*}
\end{theorem}
\begin{proof}
By Corollary \ref{C1}, we have
\begin{align}
\sum_{k=0}^{p-1}C_{2k}\,\frac{(a-b)^{2k}}{(-64ab)^k}&\equiv
\frac{1}{2}\left(\frac{-1}{p}\right)\left(3w_{\lfloor\frac{p}{4}\rfloor}\left(
\frac{a^2+b^2}{2ab}\right)-w_{\lfloor\frac{3p}{4}\rfloor}\left(\frac{a^2+b^2}{2ab}\right)\right) \pmod{p}, \label{C2k} \\
\sum_{k=0}^{p-1}{4k\choose 2k}\frac{(a+b)^{2k}}{(64ab)^k}&\equiv
\frac{1}{4}\left(\frac{-2}{p}\right)\left(3w_{\lfloor\frac{p}{4}\rfloor}\left(
\frac{a^2+b^2}{2ab}\right)+w_{\lfloor\frac{3p}{4}\rfloor}\left(\frac{a^2+b^2}{2ab}\right)\right) \pmod{p}. \label{4k2k}
\end{align}
  From (\ref{eq06}) we obtain
$$
w_n\left(\frac{a^2+b^2}{2ab}\right)=\frac{a(a/b)^n-b(b/a)^n}{a-b}.
$$
If $p\equiv 1$ (mod $4$), then we have modulo $p$,
\begin{align*}
w_{\lfloor\frac{p}{4}\rfloor}\left(\frac{a^2+b^2}{2ab}\right)&=
w_{\frac{p-1}{4}}\left(\frac{a^2+b^2}{2ab}\right)=
\frac{a(a/b)^{\frac{p-1}{4}}-b(b/a)^{\frac{p-1}{4}}}{a-b}\equiv
\frac{a\bigl(\frac{b}{p}\bigr)-b\bigl(\frac{a}{p}\bigr)}{a-b}\,(ab)^{\frac{p-1}{4}},
 \\
w_{\lfloor\frac{3p}{4}\rfloor}\left(\frac{a^2+b^2}{2ab}\right)&=
\frac{a(a/b)^{\frac{3(p-1)}{4}}-b(b/a)^{\frac{3(p-1)}{4}}}{a-b}\equiv
\frac{a\bigl(\frac{a}{p}\bigr)-b\bigl(\frac{b}{p}\bigr)}{a-b}\,(ab)^{\frac{p-1}{4}}.
\end{align*}
If $p\equiv 3$ (mod $4$), then
\begin{align*}
w_{\lfloor\frac{p}{4}\rfloor}\left(\frac{a^2+b^2}{2ab}\right)&=
w_{\frac{p-3}{4}}\left(\frac{a^2+b^2}{2ab}\right)=
\frac{a(a/b)^{\frac{p-3}{4}}-b(b/a)^{\frac{p-3}{4}}}{a-b}\equiv
\frac{\bigl(\frac{b}{p}\bigr)-\bigl(\frac{a}{p}\bigr)}{a-b}\,(ab)^{\frac{p+1}{4}},
 \\
w_{\lfloor\frac{3p}{4}\rfloor}\left(\frac{a^2+b^2}{2ab}\right)&=
\frac{a(a/b)^{\frac{3p-1}{4}}-b(b/a)^{\frac{3p-1}{4}}}{a-b}\equiv
\frac{\frac{a}{b}\bigl(\frac{a}{p}\bigr)-\frac{b}{a}\bigl(\frac{b}{p}\bigr)}{a-b}\,(ab)^{\frac{p+1}{4}}.
\end{align*}
Now substituting the above congruences in (\ref{C2k}) and (\ref{4k2k}), we conclude the proof.
\end{proof}

\section{Congruences involving second-order  Catalan numbers}
\label{Section4}

In this section, we will deal with a particular case of Theorem \ref{TM} when $m=3$. This case leads to congruences containing
second-order Catalan numbers $C_n^{(2)}$ (sequence \seqnum{A001764} in the OEIS \cite{OEIS})  and binomial
coefficients~${3n\choose n}$ (sequence \seqnum{A005809}).
\begin{theorem} \label{T2}
Let $p$ be a prime greater than $3$, and let $t\in D_p$. Then
\begin{align}
\sum_{k=0}^{\lfloor p/3\rfloor}C_k^{(2)}t^k&\equiv 3\left(\frac{p}{3}\right) w_{\lfloor\frac{p}{3}\rfloor}(1-27t/2) \pmod{p}, \label{eq13}\\
\sum_{k=(p-1)/2}^{\lfloor 2p/3\rfloor}C_k^{(2)}t^k&\equiv -\left(\frac{p}{3}\right)\Bigl(w_{\lfloor\frac{2p}{3}\rfloor}(1-27t/2)+w_{\lfloor\frac{p}{3}\rfloor}(1-27t/2)\Bigr) \pmod{p}, \nonumber \\
\sum_{k=0}^{\lfloor p/3\rfloor}{3k\choose k} t^k&\equiv \left(\frac{p}{3}\right) w_{\lfloor\frac{p}{3}\rfloor}(27t/2-1) \pmod{p}, \label{eq15}\\
\sum_{k=(p+1)/2}^{\lfloor 2p/3\rfloor}{3k\choose k} t^k&\equiv \frac{1}{3}\left(\frac{p}{3}\right)\Bigl(w_{\lfloor\frac{2p}{3}\rfloor}(27t/2-1)-w_{\lfloor\frac{p}{3}\rfloor}(27t/2-1)\Bigr) \pmod{p}. \label{eq16}
\end{align}
\end{theorem}
\begin{corollary} \label{C2}
Let $p$ be a prime greater than $3$, and let $t\in D_p$. Then
\begin{align*}
\sum_{k=0}^{p-1}C_k^{(2)}t^k&\equiv \left(\frac{p}{3}\right)\Bigl(2w_{\lfloor\frac{p}{3}\rfloor}(1-27t/2)-w_{\lfloor\frac{2p}{3}\rfloor}(1-27t/2)\Bigr) \pmod{p}, \\
\sum_{k=0}^{p-1}{3k\choose k} t^k&\equiv \frac{1}{3}\left(\frac{p}{3}\right)\Bigl(2w_{\lfloor\frac{p}{3}\rfloor}(27t/2-1)+w_{\lfloor\frac{2p}{3}\rfloor}(27t/2-1)\Bigr) \pmod{p}.
\end{align*}
\end{corollary}
Using the exact values of $w_n$ from (\ref{w-101}) and (\ref{12-1}),
we immediately get numerical congruences at the points $t=4/27, 2/27,
1/27, 1/9$.
\begin{corollary}
Let $p$ be a prime greater than $3$. Then
\begin{align}
\sum_{k=0}^{p-1}C_{k}^{(2)}\left(\frac{4}{27}\right)^k&\equiv 1,
\qquad\quad
\sum_{k=0}^{\lfloor p/3\rfloor}C_{k}^{(2)}\left(\frac{4}{27}\right)^k\equiv 3
\pmod{p}, \nonumber \\
\sum_{k=0}^{p-1}{3k\choose k}\left(\frac{4}{27}\right)^k&\equiv
\frac{1}{9}, \qquad\quad
\sum_{k=0}^{\lfloor p/3\rfloor}{3k\choose k}\left(\frac{4}{27}\right)^k\equiv
\frac{1}{3} \pmod{p}, \label{427}
\end{align}
\begin{align}
\sum_{k=0}^{p-1}C_{k}^{(2)}\left(\frac{2}{27}\right)^k&\equiv
2\left(\frac{3}{p}\right)-1,
\qquad\quad
\sum_{k=0}^{\lfloor p/3\rfloor}C_{k}^{(2)}\left(\frac{2}{27}\right)^k\equiv
3\left(\frac{3}{p}\right)
\pmod{p}, \nonumber \\
\sum_{k=0}^{p-1}{3k\choose k}\left(\frac{2}{27}\right)^k&\equiv
\frac{2}{3}\left(\frac{3}{p}\right)+\frac{1}{3}, \qquad\quad
\sum_{k=0}^{\lfloor p/3\rfloor}{3k\choose k}\left(\frac{2}{27}\right)^k\equiv
\left(\frac{3}{p}\right) \pmod{p}. \label{227}
\end{align}
\end{corollary}
\begin{remark}
 Z.\ W.\ Sun \cite[Thm.\ 3.1]{ZWS09} gave another proof of the first congruence in (\ref{427})
based on third-order recurrences. Z.\ H.\ Sun \cite[Rem.\ 3.1]{ZHS14} proved the second
congruence in (\ref{427}) as well as the second congruence in (\ref{227})
with the help of Lucas sequences.
\end{remark}
\begin{corollary}
Let $p$ be a prime, $p>3$. Then
\begin{align*}
\sum_{k=0}^{\lfloor p/3\rfloor}\frac{C_k^{(2)}}{27^k}&\equiv
\begin{cases}
\ds
-6 \pmod{p},
&   \text{if $p\equiv \pm 4 \pmod{9}$;} \\
\ds 3 \pmod{p}, &   \text{otherwise,}
\end{cases}
\\[3pt]
\sum_{k=0}^{p-1}\frac{C_k^{(2)}}{27^k}&
\equiv
\begin{cases}
\ds
1 \pmod{p},
&  \text{if $p\equiv \pm 1 \pmod{9}$;} \\
\ds 4 \pmod{p}, &   \text{if $p\equiv \pm 2 \pmod{9}$;} \\
\ds -5 \pmod{p}, &  \text{if $p\equiv \pm 4 \pmod{9}$,}
\end{cases}
\\[3pt]
\sum_{k=0}^{p-1}{3k\choose k}\frac{1}{27^k}&
\equiv
\begin{cases}
\ds
1 \pmod{p},
&  \text{if $p\equiv \pm 1 \pmod{9}$;} \\
\ds -2/3 \pmod{p}, &   \text{if $p\equiv \pm 2 \pmod{9}$;} \\
\ds -1/3 \pmod{p}, &   \text{if $p\equiv \pm 4 \pmod{9}$.}
\end{cases}
\end{align*}
\end{corollary}
\begin{corollary}
Let $p$ be a prime, $p>3$. Then
\begin{align}
\sum_{k=0}^{p-1}\frac{C_k^{(2)}}{9^k}&\equiv
\begin{cases}
\ds
-2 \pmod{p},
&  \text{if $p\equiv \pm 2 \pmod{9}$;} \\
\ds 1 \pmod{p}, &   \text{otherwise,}
\end{cases}
\label{C9}\\[3pt]
\sum_{k=0}^{\lfloor p/3\rfloor}\frac{C_k^{(2)}}{9^k}&
\equiv
\begin{cases}
\ds
3 \pmod{p},
&  \text{if $p\equiv \pm 1 \pmod{9}$;} \\
\ds -3 \pmod{p}, &   \text{if $p\equiv \pm 2 \pmod{9}$;} \\
\ds 0 \pmod{p}, &   \text{if $p\equiv \pm 4 \pmod{9}$,}
\end{cases}
\nonumber \\[3pt]
\sum_{k=0}^{p-1}{3k\choose k}\frac{1}{9^k}&
\equiv
\begin{cases}
\ds
1 \pmod{p},
&  \text{if $p\equiv \pm 1 \pmod{9}$;} \\
\ds 0 \pmod{p}, &   \text{if $p\equiv \pm 2 \pmod{9}$;} \\
\ds -1 \pmod{p}, &   \text{if $p\equiv \pm 4 \pmod{9}$.}
\end{cases}
\label{319}
\end{align}
\end{corollary}
\begin{remark}
 Z.\ W.\ Sun \cite[Thm.\ 1.5]{ZWS09} provided another proof of congruences (\ref{C9}) and  (\ref{319})
  by using cubic residues and third-order recurrences.
\end{remark}
\begin{lemma} \label{L4.1}
For any $x\ne 1, -1/2$, we have
$$
w_n(4x^3-3x)=
\frac{\alpha^{3n+2}-\alpha^{-3n-1}}{\alpha^2-\alpha^{-1}}, \qquad
\text{where} \quad \alpha=x+\sqrt{x^2-1}.
$$
\end{lemma}
\begin{proof}
Starting with $w_{3n}(x)$, by (\ref{eq06}),
 we get
\begin{equation*}
\begin{split}
w_{3n}(x)&=
\frac{(\alpha+1)\alpha^{3n}-(\alpha^{-1}+1)\alpha^{-3n}}{\alpha-\alpha^{-1}} \\
&=
\frac{(\alpha^3+1)\alpha^{3n}-(\alpha^{-3}+1)\alpha^{-3n}}{\alpha^3-\alpha^{-3}}\cdot\frac{\alpha^3-\alpha^{-3}}{\alpha-\alpha^{-1}} \\
&+
\frac{\alpha^{3n+1}-\alpha^{3n+3}-\alpha^{-3n-1}+\alpha^{-3n-3}}{\alpha-\alpha^{-1}} \\
&=w_n(4x^3-3x)(\alpha^2+1+\alpha^{-2}) \\[3pt]
&+
\frac{\alpha^{3n+1}-\alpha^{3n+3}-\alpha^{-3n-1}+\alpha^{-3n-3}}{\alpha-\alpha^{-1}}.
\end{split}
\end{equation*}
Comparing the right and left-hand sides, we obtain
$$
w_n(4x^3-3x)(\alpha^2+1+\alpha^{-2})=
\frac{\alpha^{3n}+\alpha^{3n+3}-\alpha^{-3n}
-\alpha^{-3n-3}}{\alpha-\alpha^{-1}}=
\frac{(\alpha^3+1)(\alpha^{3n}-\alpha^{-3n-3})}{\alpha-\alpha^{-1}}
$$
and therefore,
$$
w_n(4x^3-3x)=
\frac{(\alpha^3+1)(\alpha^{3n}-\alpha^{-3n-3})}{\alpha^3-\alpha^{-3}}
=\frac{\alpha^3(\alpha^{3n}-\alpha^{-3n-3})}{\alpha^3-1}=
\frac{\alpha^{3n+2}-\alpha^{-3n-1}}{\alpha^2-\alpha^{-1}}.
$$
\end{proof}
\begin{lemma} \label{L4.2}
Let $p$ be a prime, $p>3$, and let $x\in D_p$. Then
\begin{align*}
(2x+1)\cdot w_{\lfloor\frac{p}{3}\rfloor}(4x^3-3x)&\equiv \left(\frac{p}{3}\right)x+
\left(\frac{x^2-1}{p}\right)(x+1)  \pmod{p}, \\
(2x+1)\cdot w_{\lfloor\frac{2p}{3}\rfloor}(4x^3-3x)&\equiv \left(\frac{p}{3}\right)(1-2x^2)+2\left(\frac{x^2-1}{p}\right)x(x+1)
\pmod{p}.
\end{align*}
\end{lemma}
\begin{proof}
First we suppose that $x\not\equiv 1, -1/2 \pmod{p}$.
Then by Lemma \ref{L4.1}, if $p\equiv 1$ (mod $3$), we have
\begin{equation} \label{wp31}
w_{\lfloor\frac{p}{3}\rfloor}(4x^3-3x)=w_{\frac{p-1}{3}}(4x^3-3x)=
\frac{\alpha^{p+1}-\alpha^{-p}}{\alpha^2-\alpha^{-1}}.
\end{equation}
For  $p$-powers of $\alpha$ and $\alpha^{-1}$, we easily obtain
\begin{align} \label{ap}
\alpha^{\pm p}=(x\pm \sqrt{x^2-1})^p&\equiv x^p\pm (\sqrt{x^2-1})^p
\equiv x\pm \sqrt{x^2-1}(x^2-1)^{\frac{p-1}{2}} \nonumber \\
&\equiv x\pm\left(\frac{x^2-1}{p}\right)\sqrt{x^2-1} \pmod{p}.
\end{align}
Substituting (\ref{ap}) into (\ref{wp31}) and simplifying, we get
\begin{equation*}
\begin{split}
w_{\lfloor\frac{p}{3}\rfloor}(4x^3-3x)&\equiv
\frac{(x+\sqrt{x^2-1})\bigl(x+
\bigl(\frac{x^2-1}{p}\bigr)\sqrt{x^2-1}\bigr)-x+
\bigl(\frac{x^2-1}{p}\bigr)\sqrt{x^2-1}}{(2x+1)(x-1+\sqrt{x^2-1})} \\
&\equiv \frac{x+\bigl(\frac{x^2-1}{p}\bigr)(x+1)}{2x+1} \pmod{p}.
\end{split}
\end{equation*}
If $p\equiv 2$ (mod $3$), then, by Lemma \ref{L4.1} and (\ref{ap}), we have
\begin{equation*}
\begin{split}
w_{\lfloor\frac{p}{3}\rfloor}(4x^3-3x)&=w_{\frac{p-2}{3}}(4x^3-3x)=
\frac{\alpha^p-\alpha^{1-p}}{\alpha^2-\alpha^{-1}} \\
&\equiv
\frac{x+\bigl(\frac{x^2-1}{p}\bigr)\sqrt{x^2-1}-
(x+\sqrt{x^2-1})\bigl(x-
\bigl(\frac{x^2-1}{p}\bigr)\sqrt{x^2-1}\bigr)}{(2x+1)(x-1+
\sqrt{x^2-1})} \\
&\equiv \frac{-x+\bigl(\frac{x^2-1}{p}\bigr)(x+1)}{2x+1} \pmod{p}
\end{split}
\end{equation*}
and the first congruence of the lemma follows.
Similarly, if $p\equiv 1$ (mod $3$), then we have
\begin{equation*}
\begin{split}
w_{\lfloor\frac{2p}{3}\rfloor}(4x^3-3x)&=w_{\frac{2(p-1)}{3}}(4x^3-3x)=
\frac{\alpha^{2p}-\alpha^{1-2p}}{\alpha^2-\alpha^{-1}} \\
&\equiv
\frac{\bigl(x+\bigl(\frac{x^2-1}{p}\bigr)\sqrt{x^2-1}\bigr)^2-
(x+\sqrt{x^2-1})\bigl(x-
\bigl(\frac{x^2-1}{p}\bigr)\sqrt{x^2-1}\bigr)^2}{(2x+1)(x-1
+\sqrt{x^2-1})} \pmod{p}.
\end{split}
\end{equation*}
Simplifying, we easily find
$$
w_{\lfloor\frac{2p}{3}\rfloor}(4x^3-3x)\equiv \frac{1-2x^2+2\bigl(\frac{x^2-1}{p}\bigr)x(x+1)}{2x+1} \pmod{p}.
$$
If $p\equiv 2$ (mod $3$), then
\begin{equation*}
\begin{split}
w_{\lfloor\frac{2p}{3}\rfloor}(4x^3-3x)&=w_{\frac{2p-1}{3}}(4x^3-3x)=
\frac{\alpha^{2p+1}-\alpha^{-2p}}{\alpha^2-\alpha^{-1}} \\
&\equiv
\frac{(x+\sqrt{x^2-1})\bigl(x+
\bigl(\frac{x^2-1}{p}\bigr)\sqrt{x^2-1}\bigr)^2-
\bigl(x-\bigl(\frac{x^2-1}{p})\sqrt{x^2-1}\bigr)^2}{(2x+1)(x-1+
\sqrt{x^2-1})} \pmod{p},
\end{split}
\end{equation*}
and after simplification we get
$$
w_{\lfloor\frac{2p}{3}\rfloor}(4x^3-3x)\equiv \frac{2x^2-1+2\bigl(\frac{x^2-1}{p}\bigr)x(x+1)}{2x+1} \pmod{p},
$$
as desired.

Finally, if $x\equiv 1$ (mod $p$), then, by (\ref{eq06}),
we have
$3w_{\lfloor\frac{p}{3}\rfloor}(1)=3(2\lfloor p/3\rfloor+1)\equiv (\frac{p}{3})$ (mod $p$) and
$3w_{\lfloor\frac{2p}{3}\rfloor}(1)=3(2\lfloor 2p/3\rfloor+1)\equiv -(\frac{p}{3})$ (mod $p$),
which coincide with the right-hand sides of the required congruences when $x\equiv 1$ (mod $p$).

If $x\equiv -1/2$ (mod $p$), then the congruences become trivial  and the proof is complete.
\end{proof}
\begin{lemma} \label{L4.3}
Let $p$ be a prime, $p>3$,  and let $x\in D_p$. Then we have modulo $p$,
\begin{align*}
(2x+1)\cdot w_{\lfloor\frac{p}{6}\rfloor}(4x^3-3x)&\equiv \left(\frac{2x-2}{p}\right)x+
\left(\frac{-6x-6}{p}\right)(x+1),   \\
(2x+1)\cdot w_{\lfloor\frac{5p}{6}\rfloor}(4x^3-3x)&\equiv \!\left(\frac{2x-2}{p}\right)x(4x^2+2x-1)\!-\!\left(\frac{-6x-6}{p}\right)(x+1)(4x^2-2x-1).
\end{align*}
\end{lemma}
\begin{proof}
First we suppose that $x\not\equiv 1, -1/2$ (mod $p$).
If $p\equiv 1$ (mod $6$), then, by Lemma \ref{L4.1},  we have
\begin{equation*}
w_{\lfloor\frac{p}{6}\rfloor}(4x^3-3x)=w_{\frac{p-1}{6}}(4x^3-3x)=
\frac{\alpha^{\frac{p+1}{2}+1}-\alpha^{-\frac{p+1}{2}}}{\alpha^2-\alpha^{-1}}.
\end{equation*}
Substituting
$\alpha=x+\sqrt{x^2-1}=(\sqrt{(x+1)/2}+\sqrt{(x-1)/2})^2$,
we have
\begin{equation*}
\begin{split}
w_{\lfloor\frac{p}{6}\rfloor}(4x^3-3x)&=
\frac{(x+\sqrt{x^2-1})(\sqrt{x+1}+\sqrt{x-1})^{p+1}-
(\sqrt{x+1}-\sqrt{x-1})^{p+1}}{2^{(p+1)/2}(\alpha^2-\alpha^{-1})} \\[3pt]
&=\frac{(x+\sqrt{x^2-1})(\sqrt{x+1}+\sqrt{x-1})^p
-1/2(\sqrt{x+1}-\sqrt{x-1})^{p+2}}{2^{(p+1)/2}(2x+1)\sqrt{x-1}} \\[3pt]
&\equiv
\frac{(x+\sqrt{x^2-1})((x+1)^{\frac{p}{2}}+(x-1)^{\frac{p}{2}})-\!
(x-\sqrt{x^2-1})((x+1)^{\frac{p}{2}}
\!-\!(x-1)^{\frac{p}{2}})}{2^{(p+1)/2}(2x+1)\sqrt{x-1}} \\[3pt]
&\equiv \frac{x(x-1)^{\frac{p-1}{2}}+(x+1)^{\frac{p+1}{2}}}{2^{(p-1)/2}(2x+1)}
\equiv \frac{\bigl(\frac{2x-2}{p}\bigr)x+\bigl(\frac{2x+2}{p}\bigr)(x+1)}{2x+1}
\pmod{p}.
\end{split}
\end{equation*}
Since $(\frac{-3}{p})=(\frac{p}{3})=1$, we get the desired congruence in this case.

If $p\equiv 5$ (mod $6$), then we have
\begin{equation*}
w_{\lfloor\frac{p}{6}\rfloor}(4x^3-3x)=
w_{\frac{p-5}{6}}(4x^3-3x)=
\frac{\alpha^{\frac{p-1}{2}}-\alpha^{-\frac{p-3}{2}}}{\alpha^2-\alpha^{-1}}
\end{equation*}
and therefore,
\begin{equation*}
\begin{split}
w_{\lfloor\frac{p}{6}\rfloor}(4x^3-3x)
&\equiv
\frac{(x-\sqrt{x^2-1})((x+1)^{\frac{p}{2}}+(x-1)^{\frac{p}{2}})-\!
(x+\sqrt{x^2-1})((x+1)^{\frac{p}{2}}
\!-\!(x-1)^{\frac{p}{2}})}{2^{(p+1)/2}(2x+1)\sqrt{x-1}} \\[3pt]
&\equiv \frac{x(x-1)^{\frac{p-1}{2}}-(x+1)^{\frac{p+1}{2}}}{2^{(p-1)/2}(2x+1)}
\equiv \frac{\bigl(\frac{2x-2}{p}\bigr)x-\bigl(\frac{2x+2}{p}\bigr)(x+1)}{2x+1}
\pmod{p},
\end{split}
\end{equation*}
as desired in view of the fact that $(\frac{-3}{p})=(\frac{p}{3})=-1$.

 The similar analysis can be applied for evaluating $w_{\lfloor\frac{5p}{6}\rfloor}(4x^3-3x)$ modulo $p$. If $p\equiv 1$ (mod $6$), then
\begin{equation*}
w_{\lfloor\frac{5p}{6}\rfloor}(4x^3-3x)=
w_{\frac{5(p-1)}{6}}(4x^3-3x)=
\frac{\alpha^{\frac{5p-1}{2}}-\alpha^{-\frac{5p-3}{2}}}{\alpha^2-\alpha^{-1}}.
\end{equation*}
Simplifying, we obtain
\begin{equation*}
\begin{split}
w_{\lfloor\frac{5p}{6}\rfloor}(4x^3&-3x)=
\frac{(x-\sqrt{x^2-1})(\sqrt{x+1}+\sqrt{x-1})^{5p}
\!-\!(x+\sqrt{x^2-1})(\sqrt{x+1}-\sqrt{x-1})^{5p}}{2^{(5p+1)/2} (2x+1)\sqrt{x-1}} \\[3pt]
&\equiv
\frac{(x-\sqrt{x^2-1})((x+1)^{\frac{p}{2}}+(x-1)^{\frac{p}{2}})^5-
(x+\sqrt{x^2-1})((x+1)^{\frac{p}{2}}
-(x-1)^{\frac{p}{2}})^5}{2^{(5p+1)/2}(2x+1)\sqrt{x-1}} \\[3pt]
&\equiv  \frac{\bigl(\frac{2x-2}{p}\bigr)x(4x^2+2x-1)-\bigl(\frac{2x+2}{p}\bigr)(x+1)(4x^2-2x-1)}{2x+1}
\pmod{p},
\end{split}
\end{equation*}
as desired. If $p\equiv 5$ (mod $6$), then
\begin{equation*}
\begin{split}
w_{\lfloor\frac{5p}{6}\rfloor}(4x^3&-3x)=w_{\frac{5p-1}{6}}(4x^3-3x)=
\frac{\alpha^{\frac{5p+3}{2}}
-\alpha^{-\frac{5p+1}{2}}}{\alpha^2-\alpha^{-1}} \\[3pt]
&=\frac{(\sqrt{x+1}+\sqrt{x-1})^{5p+2}
-(\sqrt{x+1}-\sqrt{x-1})^{5p+2}}{2^{(5p+3)/2}(2x+1)\sqrt{x-1}} \\[3pt]
&\equiv
\frac{(x+\sqrt{x^2-1})((x+1)^{\frac{p}{2}}+(x-1)^{\frac{p}{2}})^5-
(x-\sqrt{x^2-1})((x+1)^{\frac{p}{2}}
-(x-1)^{\frac{p}{2}})^5}{8\cdot 2^{(p-1)/2}(2x+1)\sqrt{x-1}} \\[3pt]
&\equiv  \frac{\bigl(\frac{2x-2}{p}\bigr)x(4x^2+2x-1)+\bigl(\frac{2x+2}{p}\bigr)(x+1)(4x^2-2x-1)}{2x+1}
\pmod{p},
\end{split}
\end{equation*}
and the congruence is true.  If $x\equiv 1$ (mod $6$), then, by (\ref{eq06}), we have
$3w_{\lfloor\frac{p}{6}\rfloor}(1)=3(2\lfloor p/6\rfloor+1)=2(\frac{p}{3})$ (mod $p$) and
$3w_{\lfloor\frac{5p}{6}\rfloor}(1)=3(2\lfloor 5p/6\rfloor+1)=-2(\frac{p}{3})$ (mod $p$),
which prove the lemma in this case too. Finally, if $x\equiv -1/2$ (mod $p$), we get the trivial congruences $0\equiv 0$, and the proof is complete.
\end{proof}
\begin{theorem}
Let $p$ be a prime, $p>3$,  and let $t\in D_p$.

If $t\not\equiv 0 \pmod{p}$, then
\begin{align}
\sum_{k=1}^{\lfloor p/3\rfloor}C_k^{(2)}\bigl(t^{2}(t+1)\bigr)^k&\equiv \frac{1+t}{2t}
-\frac{1-3t}{2t}\left(\frac{(1+t)(1-3t)}{p}\right) \pmod{p}, \nonumber \\
\sum_{k=1}^{p-1}C_k^{(2)}\bigl(t^2(t+1)\bigr)^k&\equiv
\frac{(1+t)(1-3t)}{2t}\left(1-\left(\frac{(1+t)(1-3t)}{p}\right)\right) \pmod{p}. \label{111}
\end{align}
If $3t+2\not\equiv 0 \pmod{p}$, then
\begin{align}
\sum_{k=1}^{\lfloor p/3\rfloor}{3k\choose k}\bigl(t^2(t+1)\bigr)^k&\equiv
\frac{3(t+1)}{2(3t+2)}\left(\left(\frac{(1+t)(1-3t)}{p}\right)-1\right) \pmod{p}, \label{112} \\
\sum_{k=1}^{p-1}{3k\choose k}\bigl(t^2(t+1)\bigr)^k&\equiv
\frac{3(t+1)^2}{2(3t+2)}\left(\left(\frac{(1+t)(1-3t)}{p}\right)-1\right) \pmod{p}. \label{113}
\end{align}
\end{theorem}
\begin{proof}
  From (\ref{eq13}), Corollary \ref{C2} and Lemma \ref{L4.2} we have modulo $p$,
$$
\sum_{k=0}^{\lfloor p/3\rfloor}C_k^{(2)}\left(\frac{2(1-x)(2x+1)^2}{27}\right)^k
\!\equiv 3\left(\frac{p}{3}\right)w_{\lfloor\frac{p}{3}\rfloor}(4x^3-3x)\equiv
\frac{3x}{2x+1}+\frac{3x+3}{2x+1}\left(\frac{3-3x^2}{p}\right)
$$
and
\begin{equation*}
\begin{split}
\sum_{k=0}^{p-1}C_k^{(2)}\left(\frac{2(1-x)(2x+1)^2}{27}\right)^k
&\equiv\left(\frac{p}{3}\right)\bigl(2w_{\lfloor\frac{p}{3}\rfloor}(4x^3-3x)-
w_{\lfloor\frac{2p}{3}\rfloor}(4x^3-3x)\bigr) \\
&\equiv \frac{2x^2+2x-1}{2x+1}+\frac{2(1-x^2)}{2x+1}\left(
\frac{3-3x^2}{p}\right) \pmod{p}
\end{split}
\end{equation*}
for any $x\in D_p$ such that $2x+1\not\equiv 0$ (mod $p$). Replacing $x$ by $(-1-3t)/2$ with $t\not\equiv 0$ (mod~$p$), we get the first two congruences of the theorem.

Similarly, from (\ref{eq15}), Corollary \ref{C2} and Lemma \ref{L4.2}
for any $x\in D_p$ with $2x+1\not\equiv 0$ (mod~$p$), we have
\begin{align*}
\sum_{k=0}^{\lfloor p/3\rfloor}{3k\choose k}\!\left(\frac{2(x+1)(2x-1)^2}{27}\right)^k\!\!&\equiv \left(\frac{p}{3}\right)w_{\lfloor\frac{p}{3}\rfloor}(4x^3-3x)\equiv
\frac{x}{2x+1}+\frac{x+1}{2x+1}\left(\frac{3-3x^2}{p}\right), \\
\sum_{k=0}^{p-1}{3k\choose k}\!\left(\frac{2(x+1)(2x-1)^2}{27}\right)^k\!\!
&\equiv\frac{1}{3}\left(\frac{p}{3}\right)
\bigl(2w_{\lfloor\frac{p}{3}\rfloor}(4x^3-3x)+w_{\lfloor\frac{2p}{3}\rfloor}(4x^3-3x)\bigr) \\
&\equiv \frac{1+2x-2x^2}{3(2x+1)}+\frac{2(x+1)^2}{3(2x+1)}
\left(\frac{3-3x^2}{p}\right) \pmod{p}.
\end{align*}
This implies that
\begin{align*}
\sum_{k=1}^{\lfloor p/3\rfloor}{3k\choose k}\,\left(\frac{2(x+1)(2x-1)^2}{27}\right)^k\!\!&\equiv \frac{x+1}{2x+1}\left(\left(\frac{3-3x^2}{p}\right)-1\right)
\pmod{p}, \\
\sum_{k=1}^{p-1}{3k\choose k}\,\left(\frac{2(x+1)(2x-1)^2}{27}\right)^k\!\!&\equiv \frac{2(x+1)^2}{3(2x+1)}\left(\left(\frac{3-3x^2}{p}\right)-1\right)
\pmod{p}.
\end{align*}
Replacing $x$ by $(3t+1)/2$, we derive the other two congruences of the theorem.
\end{proof}
\begin{remark}
 Note that Z.\ H.\ Sun \cite[Thm.\ 2.3]{ZHS13} proved  congruence (\ref{112})
 by another method using cubic congruences.
If we put $t=-c/(c+1)$ in (\ref{111}) and (\ref{113}), we recover corresponding congruences of Z.\ W.\ Sun \cite[Thm.\ 1.1]{ZWS09}
proved by applying properties of third-order recurrences.
\end{remark}

\section{Cubic residues and non-residues and their application to congruences}      \label{Section5}

We begin with a brief review of basic facts from the theory of cubic residues that will be needed later in this section.
Let $\omega=e^{2\pi i/3}=(-1+i\sqrt{3})/2$. We  consider the ring of the Eisenstein  integers
${\mathbb Z}[\omega]=\{a+b\omega: a, b\in {\mathbb Z}\}$. To define the cubic residue symbol, we recall arithmetic properties of the ring
${\mathbb Z}[\omega]$
 including description of its units and primes \cite[Chapter 9]{IR}.

If $\alpha=a+b\omega\in {\mathbb Z}[\omega]$, the norm of $\alpha$ is defined by the formula $N(\alpha)=\alpha\overline{\alpha}=a^2-ab+b^2$, where
$\overline{\alpha}=a+b\overline{\omega}=a+b\omega^2=(a-b)-b\omega$
is the complex conjugate of $\alpha$. Note that the norm is a nonnegative integer always congruent to $0$ or $1$ modulo $3$.
It is well known that ${\mathbb Z}[\omega]$ is a unique factorization domain. The units of ${\mathbb Z}[\omega]$ are $\pm 1, \pm\omega, \pm\omega^2$.

Let $p$ be a prime in ${\mathbb Z}$, then $p$ in ${\mathbb Z}[\omega]$ falls into three categories \cite[Prop.\ 4.7]{Cox}: ${\rm (i)}$  if $p=3$, then
$3=-\omega^2(1-\omega)^2$, where $1-\omega$ is prime in ${\mathbb Z}[\omega]$ and $N(1-\omega)=(1-\omega)(1-\omega^2)=3$;
${\rm (ii)}$ if $p\equiv 2$ (mod $3$),
then $p$ remains prime in ${\mathbb Z}[\omega]$ and $N(p)=p^2$; ${\rm (iii)}$ if $p\equiv 1$ (mod $3$), then $p$ splits into the product of two conjugate non-associate
primes in ${\mathbb Z}[\omega]$, $p=\pi\overline{\pi}$ and $N(\pi)=\pi\overline{\pi}=p$. Moreover, every prime in ${\mathbb Z}[\omega]$ is associated with one of the primes listed in ${\rm (i)-(iii)}$.

An analog of Fermat's little theorem is true in ${\mathbb Z}[\omega]$: if $\pi$ is a prime and $\pi\nmid\alpha$, then
$$
\alpha^{N(\pi)-1}\equiv 1 \pmod{\pi}.
$$
Note that if $\pi$ is a prime such that $N(\pi)\ne 3,$ then $N(\pi)\equiv 1$ (mod $3$) and the expression $\alpha^{\frac{N(\pi)-1}{3}}$ is well defined
in ${\mathbb Z}[\omega]$, i.e.,  $\alpha^{\frac{N(\pi)-1}{3}}\equiv\omega^j$ (mod $\pi$) for a unique unit $\omega^j$. This leads to the definition
of the {\sl cubic residue character} of $\alpha$ modulo $\pi$ \cite[p.\ 112]{IR}:
\begin{equation} \label{tr0}
\left(\frac{\alpha}{\pi}\right)_3=\begin{cases}
\ds
0,
&  \text{if $\pi|\alpha$;} \\[3pt]
\omega^j, &   \text{if $\alpha^{\frac{N(\pi)-1}{3}}\equiv \omega^j \pmod{\pi}$.}
\end{cases}
\end{equation}
The cubic residue character has formal properties similar to those of the Legendre symbol~\cite[Prop.\ 9.3.3]{IR}:

${(i)}\,$ The congruence $x^3\equiv\alpha$ (mod $\pi$) is solvable in
${\mathbb Z}[\omega]$ if and only if $\left(\frac{\alpha}{\pi}\right)_3=1$, i.e., iff $\alpha$ is a cubic residue modulo $\pi$;

${(ii)}\,$ $\left(\frac{\alpha\beta}{\pi}\right)_3=\left(\frac{\alpha}{\pi}\right)_3\left(\
\frac{\beta}{\pi}\right)_3$;

${(iii)}\,$ $\overline{\left(\frac{\alpha}{\pi}\right)}_3=\left(\frac{\overline{\alpha}}{\overline{\pi}}\right)_3$;

${(iv)}\,$ If $\pi$ and $\theta$ are associates, then $\left(\frac{\alpha}{\pi}\right)_3=\left(\frac{\alpha}{\theta}\right)_3$;

${(v)}\,$ If $\alpha\equiv\beta$ (mod $\pi$), then $\left(\frac{\alpha}{\pi}\right)_3=\left(\frac{\beta}{\pi}\right)_3$.

Let $\pi=a+b\omega\in {\mathbb Z}[\omega]$. We say that $\pi$ is {\sl primary} if $\pi\equiv 2$ (mod $3$), that is equivalent to $a\equiv 2$ (mod $3$)
and $b\equiv 0$ (mod $3$).
If $\pi\in {\mathbb Z}[\omega]$,  $N(\pi)>1$ and $\pi\equiv\pm 2$ (mod $3$), we may decompose $\pi=\pm\pi_1\dots\pi_r$, where $\pi, \dots, \pi_r$ are primary primes \cite[p.\ 135]{IR}.  For $\alpha\in {\mathbb Z}[\omega]$,
the {\sl cubic Jacobi symbol} $\left(\frac{\alpha}{\pi}\right)_3$ is defined by
$$
\left(\frac{\alpha}{\pi}\right)_3=
\left(\frac{\alpha}{\pi_1}\right)_3\dots \left(\frac{\alpha}{\pi_r}\right)_3.
$$
Now let $p$ be a prime. We define a cubic residue modulo $p$ in ${\mathbb Z}$. We say that $m\in {\mathbb Z}$ is a {\sl cubic residue modulo} $p$ if the congruence $x^3\equiv m$ (mod $p$) has an integer solution, otherwise $m$ is called a {\sl cubic non-residue modulo} $p$.
If $p=3,$ then by Fermat's little theorem, $m^3\equiv m$ (mod $3$)
for all integers $m,$ so $x^3\equiv m$ (mod $3$) always has a solution.
If $p\equiv 2$ (mod $3$), then every integer $m$ is a cubic residue modulo $p$. Indeed, we have $2p-1\equiv 0$ (mod $3$)
and by Fermat's little theorem,
 $m\equiv m^{2p-1}=\bigl(m^{\frac{2p-1}{3}}\bigr)^3$ (mod $p$).
So the only interesting case which remains is when  a prime $p\equiv 1$ (mod $3$).

If a prime $p\equiv 1$ (mod $3$), then it is well known that there are unique integers $L$ and $|M|$ such that $4p=L^2+27M^2$ with $L\equiv 1$ (mod $3$).
In this case, $p$ splits into the product of primes of
${\mathbb Z}[\omega]$, $p=\pi\overline{\pi}$, where we can write $\pi$ in the form
$$
\pi=\frac{1}{2}(L+3M\sqrt{-3})=\frac{L+3M}{2}+3M\omega.
$$
It is easy to see that $\left(\frac{L}{3M}\right)^2\equiv -3$ (mod $p$)
and therefore for any integer $m$ coprime to $p$
by Euler's criterion \cite{L59, W75},
 we have one of the three possibilities
$$
m^{(p-1)/3}\equiv 1, \quad (-1-L/(3M))/2 \quad\text{or}\quad
(-1+L/(3M))/2 \pmod{p}.
$$
Moreover, $m^{(p-1)/3}\equiv 1$ (mod $p$) if and only if $m$ is a cubic residue modulo $p$. When $m$ is a prime and a cubic non-residue modulo $p$,
Williams \cite{W75} found a method how to choose the sign of $M$ so that $m^{(p-1)/3}\equiv (-1-L/(3M))/2$ (mod $p$).
To classify cubic residues and non-residues in ${\mathbb Z}$,  Sun \cite{ZHS98} introduced three subsets
$$
C_j(m)=\left\{c\in D_m\left|\left(\frac{c+1+2\omega}{m}\right)_3=\omega^j\right.\right\},
\quad j=0, 1, 2, \,\, m\in{\mathbb N}, \, m\not\equiv 0\pmod{3},
$$
of $D_m,$ which posses the following properties:

\vspace{0.2cm}

${\rm (i)} \,\, C_0(m)\cup C_1(m)\cup C_2(m)=\{c\in D_m\, |\,\, \gcd(c^2+3, m)=1\}$;

\vspace{0.1cm}

${\rm (ii)} \,\, c\in C_0(m) \,\,\,\text{if and only if}\,\,\,
-c\in C_0(m)$;

\vspace{0.1cm}

${\rm (iii)} \,\, c\in C_1(m) \,\,\,\text{if and only if}\,\,\,
-c\in C_2(m)$;

\vspace{0.1cm}

${\rm (iv)} \,\, \text{If} \,\,\, c, c'\in D_m\,\,\,\text{and} \,\,\,
cc'\equiv -3\pmod{m}, \,\, \text{then} \,\,\, c\in C_j(m)   \,\,\,\text{if and only if}\,\,\,
c'\in C_j(m)$.

\vspace{0.1cm}

Using these sets, Z.\ H.\ Sun   proved the following criterion of cubic residuacity in ${\mathbb Z}$: {\it Let $p$ be a prime of the form
$p\equiv 1 \pmod{3}$ and hence $4p=L^2+27M^2$ for some $L, M\in {\mathbb Z}$ and $L\equiv 1 \pmod{3}$. If $q$ is a prime with $q|M$, then
$q^{(p-1)/3}\equiv 1 \pmod{p}$. If $q\nmid M$ and $j\in\{0,1,2\}$, then
\begin{equation} \label{cr}
q^{(p-1)/3}\equiv ((-1-L/(3M))/2)^j \!\!\!\pmod{p} \,\,\,
\text{if and only if} \,\,\, L/(3M)\in C_j(q).
\end{equation} }
Sun \cite{ZHS13} gave a simple criterion  in terms of values of the sum $\sum_{k=1}^{\lfloor p/3\rfloor}{3k\choose k}
\left(\frac{4}{9(c^2+3)}\right)^k$ modulo a prime  $p$ for $c\in C_j(p)$ and conjectured a similar criterion in terms of the sum
 $\sum_{k=(p+1)/2}^{\lfloor 2p/3\rfloor}{3k\choose k}t^k$.

 In this section, using our formulas from Theorem \ref{T2},
we address this question of Sun (see Theorem~\ref{MqL} below).
First,  we will need the following statement.
\begin{lemma}{\rm(\cite[Lemma 2.2]{ZHS98})} \label{LR}
Let $p$ be a prime, $p\ne 3$, and let $c\in D_p$.

$(i)$ If $p\equiv 1\pmod{3}$ and so $p$ splits into the product of primes, $p=\pi\overline{\pi}$ with $\pi\in {\mathbb Z}[\omega]$
and $\pi\equiv 2\pmod{3}$, then
\begin{equation*}
\begin{split}
\left(\frac{c+1+2\omega}{p}\right)_3&=
\left(\frac{(c^2+3)(c-1-2\omega)}{\pi}\right)_3, \\
\left(\frac{c-1-2\omega}{p}\right)_3&=
\left(\frac{c+1+2\omega}{p}\right)_3^{-1}=
\left(\frac{(c^2+3)(c+1+2\omega)}{\pi}\right)_3.
\end{split}
\end{equation*}

$(ii)$ If $p\equiv 2\pmod{3},$ then
\begin{equation*}
\begin{split}
\left(\frac{c+1+2\omega}{p}\right)_3&\equiv
(c^2+3)^{(p-2)/3}(c+1+2\omega)^{(p+1)/3}  \pmod{p}, \\
\left(\frac{c-1-2\omega}{p}\right)_3&=
\left(\frac{c+1+2\omega}{p}\right)_3^{-1}\equiv
(c^2+3)^{(p-2)/3}(c-1-2\omega)^{(p+1)/3} \pmod{p}.
\end{split}
\end{equation*}
\end{lemma}
Now we prove the following criterion.
\begin{theorem} \label{TR}
Let $p$ be a prime, $p>3$, and let $c\in D_p$ with $c^2\not\equiv -3\pmod{p}$. Then
\begin{equation*}
c\sum_{k=(p+1)/2}^{\lfloor 2p/3\rfloor}{3k\choose k}\left(\frac{4}{9(c^2+3)}\right)^k\equiv
\begin{cases}
\ds
0 \pmod{p},
&  \text{if $c\in C_0(p)$;} \\
1 \pmod{p}, &   \text{if $c\in C_1(p)$;} \\
-1 \pmod{p}, &   \text{if $c\in C_2(p)$.}
\end{cases}
\end{equation*}
\end{theorem}
\begin{proof}
By (\ref{eq16}), we have
\begin{equation} \label{tr1}
\sum_{k=(p+1)/2}^{\lfloor 2p/3\rfloor}{3k\choose k}\left(\frac{4}{9(c^2+3)}\right)^k\equiv \frac{1}{3}\left(\frac{p}{3}\right)\left(w_{\lfloor\frac{2p}{3}\rfloor}\biggl(\frac{3-c^2}{3+c^2}\biggr)-w_{\lfloor\frac{p}{3}\rfloor}\biggl(\frac{3-c^2}{3+c^2}\biggr)\right) \pmod{p}.
\end{equation}
  From (\ref{eq06}) it easily follows that
\begin{equation} \label{tr2}
w_{n}\biggl(\frac{3-c^2}{3+c^2}\biggr)=\frac{(-1)^n}{2c(c^2+3)^n}
\left((c-1-2\omega)^{2n+1}+(c+1+2\omega)^{2n+1}\right).
\end{equation}
If $p\equiv 1$ (mod $3$), then $p$ splits into the product of primes in ${\mathbb Z}[\omega]$, $p=\pi\overline{\pi}$ with $\pi\equiv 2$ (mod~$3$) and, by (\ref{tr2}), we have
\begin{equation} \label{tr3}
w_{\lfloor\frac{p}{3}\rfloor}\biggl(\frac{3-c^2}{3+c^2}\biggr)=
\frac{1}{2c(c^2+3)^{(p-1)/3}}\left((c-1-2\omega)^{2(p-1)/3+1}+(c+1+2\omega)^{2(p-1)/3+1}\right).
\end{equation}
By (\ref{tr0}) and Lemma \ref{LR}, we have
\begin{equation} \label{tr4}
(c^2+3)^{2(p-1)/3}(c-1-2\omega)^{2(p-1)/3}\equiv \left(\frac{(c^2+3)(c-1-2\omega)}{\pi}\right)^2=\left(\frac{c+1+2\omega}{p}\right)_3^2 \pmod{\pi}
\end{equation}
and
\begin{equation} \label{tr5}
(c^2+3)^{2(p-1)/3}(c+1+2\omega)^{2(p-1)/3}\equiv \left(\frac{(c^2+3)(c+1+2\omega)}{\pi}\right)^2\!=\left(\frac{c+1+2\omega}{p}\right)_3^{-2} \!\pmod{\pi}.
\end{equation}
Substituting (\ref{tr4}) and (\ref{tr5}) into (\ref{tr3}), we get
$$
w_{\lfloor\frac{p}{3}\rfloor}\biggl(\frac{3-c^2}{3+c^2}\biggr)\!\equiv\!
\frac{1}{2c}\left(\!(c-1-2\omega)\!\left(\frac{c+1+2\omega}{p}\right)_3^2\!\!+(c+1+2\omega)\left(\frac{c+1+2\omega}{p}\right)_3^{-2}\right)\!\!\! \pmod{\pi}
$$
and therefore,
\begin{equation} \label{tr6}
w_{\lfloor\frac{p}{3}\rfloor}\biggl(\frac{3-c^2}{3+c^2}\biggr)\equiv
\begin{cases}
\ds
1 \pmod{\pi},
&  \text{if $c\in C_0(p)$;} \\[2pt]
-\frac{3+c}{2c} \pmod{\pi}, &   \text{if $c\in C_1(p)$;} \\[2pt]
\frac{3-c}{2c} \pmod{\pi}, &   \text{if $c\in C_2(p)$.}
\end{cases}
\end{equation}
Since both sides of the above congruence are rational, the congruence is also true modulo $p=\pi\overline{\pi}$.
Similarly, if $p\equiv 2$  (mod $3$), then
\begin{equation} \label{tr7}
w_{\lfloor\frac{p}{3}\rfloor}\biggl(\frac{3-c^2}{3+c^2}\biggr)=
\frac{-1}{2c(c^2+3)^{(p-2)/3}}\left((c-1-2\omega)^{2(p+1)/3-1}+
(c+1+2\omega)^{2(p+1)/3-1} \right).
\end{equation}
Now, by Lemma \ref{LR}, we have
$$
(c+1+2\omega)^{2(p+1)/3}\equiv (c^2+3)^{-2(p-2)/3}\left(\frac{c+1+2\omega}{p}\right)_3^2 \pmod{p}
$$
and
$$
(c-1-2\omega)^{2(p+1)/3}\equiv (c^2+3)^{-2(p-2)/3}\left(\frac{c+1+2\omega}{p}\right)_3^{-2} \pmod{p}.
$$
Substituting the above congruences into (\ref{tr7}) and noticing that
$c^2+3=(c+1+2\omega)(c-1-2\omega)$, we get
\begin{equation*}
w_{\lfloor\frac{p}{3}\rfloor}\biggl(\frac{3-c^2}{3+c^2}\biggr)\equiv
\frac{-1}{2c}\left((c+1+2\omega)\left(\frac{c+1+2\omega}{p}\right)_3^{-2}+
(c-1-2\omega)\left(\frac{c+1+2\omega}{p}\right)_3^{2} \right)
\end{equation*}
and therefore,
\begin{equation} \label{tr8}
w_{\lfloor\frac{p}{3}\rfloor}\biggl(\frac{3-c^2}{3+c^2}\biggr)\equiv
\begin{cases}
\ds
-1 \pmod{p},
&  \text{if $c\in C_0(p)$;} \\[2pt]
\frac{3+c}{2c} \pmod{p}, &   \text{if $c\in C_1(p)$;} \\[2pt]
-\frac{3-c}{2c} \pmod{p}, &   \text{if $c\in C_2(p)$.}
\end{cases}
\end{equation}
Combining congruences (\ref{tr6}) and (\ref{tr8}), we obtain that for all primes $p>3$,
\begin{equation} \label{tr9}
\left(\frac{p}{3}\right)w_{\lfloor\frac{p}{3}\rfloor}\biggl(\frac{3-c^2}{3+c^2}\biggr)\equiv
\begin{cases}
\ds
1 \pmod{p},
&  \text{if $c\in C_0(p)$;} \\[2pt]
-\frac{3+c}{2c} \pmod{p}, &   \text{if $c\in C_1(p)$;} \\[2pt]
\frac{3-c}{2c} \pmod{p}, &   \text{if $c\in C_2(p)$.}
\end{cases}
\end{equation}
Applying the similar argument for evaluation of
$w_{\lfloor\frac{2p}{3}\rfloor}(\frac{3-c^2}{3+c^2})$, we see that if $p\equiv 1$ (mod $3$), then $2p\equiv 2$ (mod $3$) and therefore,
\begin{equation*}
\begin{split}
w_{\lfloor\frac{2p}{3}\rfloor}&\biggl(\frac{3-c^2}{3+c^2}\biggr)=
\frac{1}{2c(c^2+3)^{2(p-1)/3}}\left((c-1-2\omega)^{4(p-1)/3+1}+
(c+1+2\omega)^{4(p-1)/3+1} \right)\\
&\equiv \frac{1}{2c}\left((c-1-2\omega)\left(\frac{c+1+2\omega}{p}\right)_3+(c+1+2\omega)\left(\frac{c+1+2\omega}{p}\right)_3^{-1} \right) \pmod{\pi},
\end{split}
\end{equation*}
which implies
\begin{equation} \label{tr10}
w_{\lfloor\frac{2p}{3}\rfloor}\biggl(\frac{3-c^2}{3+c^2}\biggr)\equiv
\begin{cases}
\ds
1 \pmod{p},
& \text{if $c\in C_0(p)$;} \\[2pt]
\frac{3-c}{2c} \pmod{p}, &   \text{if $c\in C_1(p)$;} \\[2pt]
-\frac{3+c}{2c} \pmod{p}, &   \text{if $c\in C_2(p)$.}
\end{cases}
\end{equation}
If $p\equiv 2$ (mod $3$), then $2p\equiv 1$ (mod $3$) and we have
\begin{equation*}
\begin{split}
w_{\lfloor\frac{2p}{3}\rfloor}&\biggl(\frac{3-c^2}{3+c^2}\biggr)=
\frac{-1}{2c(c^2+3)^{(2p-1)/3}}\left(
(c+1+2\omega)^{4(p+1)/3-1}+(c-1-2\omega)^{4(p+1)/3-1}\right)\\
&\equiv \frac{-1}{2c}\left((c-1-2\omega)\left(\frac{c+1+2\omega}{p}\right)_3+(c+1+2\omega)\left(\frac{c+1+2\omega}{p}\right)_3^{-1} \right) \pmod{p},
\end{split}
\end{equation*}
and therefore,
\begin{equation} \label{tr11}
w_{\lfloor\frac{2p}{3}\rfloor}\biggl(\frac{3-c^2}{3+c^2}\biggr)\equiv
\begin{cases}
\ds
-1 \pmod{p},
&  \text{if $c\in C_0(p)$;} \\[2pt]
-\frac{3-c}{2c} \pmod{p}, &   \text{if $c\in C_1(p)$;} \\[2pt]
\frac{3+c}{2c} \pmod{p}, &   \text{if $c\in C_2(p)$.}
\end{cases}
\end{equation}
Combining (\ref{tr10}) and (\ref{tr11}), we see that for all primes $p>3$,
\begin{equation} \label{tr12}
\left(\frac{p}{3}\right)w_{\lfloor\frac{2p}{3}\rfloor}\biggl(\frac{3-c^2}{3+c^2}\biggr)\equiv
\begin{cases}
\ds
1 \pmod{p},
&  \text{if $c\in C_0(p)$;} \\[2pt]
\frac{3-c}{2c} \pmod{p}, &   \text{if $c\in C_1(p)$;} \\[2pt]
-\frac{3+c}{2c} \pmod{p}, &   \text{if $c\in C_2(p)$.}
\end{cases}
\end{equation}
Now, by (\ref{tr9}), (\ref{tr12})
and (\ref{tr1}), the congruence of the theorem easily follows.
\end{proof}

  From Theorem \ref{TR} and criterion (\ref{cr}) we deduce the following result
confirming a question of Z.\ H.\ Sun~\cite[Conj.\ 2.1]{ZHS13}.
\begin{theorem} \label{MqL}
Let $q$ be a prime, $q\equiv 1 \pmod{3}$ and so $4q=L^2+27M^2$ with $L, M\in {\mathbb Z}$ and $L\equiv 1\pmod{3}$. Let $p$ be a prime with
$p\ne 2, 3, q$, and let $p\nmid LM$. Then
\begin{equation*}
\sum_{k=(p+1)/2}^{\lfloor 2p/3\rfloor}{3k\choose k}\frac{M^{2k}}{q^k}\equiv
\begin{cases}
\ds
0 \pmod{p},
&  \text{if $p^{\frac{q-1}{3}}\equiv 1\pmod{q}$;} \\[3pt]
\pm\frac{3M}{L} \pmod{p}, &   \text{if $p^{\frac{q-1}{3}}\equiv \frac{-1\pm 9M/L}{2}$} \pmod{q}
\end{cases}
\end{equation*}
and
\begin{equation*}
\sum_{k=(p+1)/2}^{\lfloor 2p/3\rfloor}{3k\choose k}\frac{L^{2k}}{(27q)^k}\equiv
\begin{cases}
\ds
0 \pmod{p},
&  \text{if $p^{\frac{q-1}{3}}\equiv 1\pmod{q}$;} \\[3pt]
\pm\frac{L}{9M} \pmod{p}, &  \text{if $p^{\frac{q-1}{3}}\equiv \frac{-1\pm L/(3M)}{2}$} \pmod{q}.
\end{cases}
\end{equation*}
\end{theorem}
\begin{proof}
To prove the first congruence, we put $c=\frac{L}{3M}$ in Theorem \ref{TR}. Then $c(c^2+3)\not\equiv 0$ (mod $p$), $\frac{4}{9(c^2+3)}=\frac{M^2}{q}$ and we have
\begin{equation*}
\sum_{k=(p+1)/2}^{\lfloor 2p/3\rfloor}{3k\choose k}\frac{M^{2k}}{q^k}\equiv
\begin{cases}
\ds
0 \pmod{p},
&  \text{if $L/(3M)\in C_0(p)$;} \\[2pt]
\frac{3M}{L} \pmod{p}, &   \text{if $L/(3M)\in C_1(p)$;} \\[2pt]
-\frac{3M}{L} \pmod{p}, &   \text{if $L/(3M)\in C_2(p)$.}
\end{cases}
\end{equation*}
Now applying (\ref{cr}) and taking into account that $L/(3M)\equiv -9M/L$ (mod $q$), we get the result.

To prove the second congruence, we put $c=-9M/L$ in Theorem \ref{TR}. Then $c(c^2+3)\not\equiv 0$ (mod $p$), $\frac{4}{9(c^2+3)}=\frac{L^2}{27q}$ and we have
\begin{equation} \label{a}
\sum_{k=(p+1)/2}^{\lfloor 2p/3\rfloor}{3k\choose k}\frac{L^{2k}}{(27q)^k}\equiv
\begin{cases}
\ds
0 \pmod{p},
&  \text{if $-9M/L\in C_0(p)$;} \\[2pt]
-\frac{L}{9M} \pmod{p}, &   \text{if $-9M/L\in C_1(p)$;} \\[2pt]
\frac{L}{9M} \pmod{p}, &   \text{if $-9M/L\in C_2(p)$.}
\end{cases}
\end{equation}
By ${\rm (iv)},$ we know that $-9M/L\in C_j(p)$ if and only if
$L/(3M)\in C_j(p)$. This together with (\ref{a}) and (\ref{cr})
implies the required congruence.
\end{proof}
From Corollary \ref{C2} and formulas (\ref{tr9})
and (\ref{tr12}) we get the following statement.

\begin{theorem} \label{CaT}
Let $p$ be a prime, $p>3$, and let $c\in D_p$ with $c^2\not\equiv -3\pmod{p}$. Then
\begin{equation*}
\sum_{k=0}^{p-1}{3k\choose k}\left(\frac{4}{9(c^2+3)}\right)^k\equiv
\begin{cases}
\ds
1 \pmod{p},
&  \text{if $c\in C_0(p)$;} \\
-\frac{1+c}{2c} \pmod{p}, &   \text{if $c\in C_1(p)$;} \\
\frac{1-c}{2c} \pmod{p}, &   \text{if $c\in C_2(p)$}
\end{cases}
\end{equation*}
and
\begin{equation*}
\sum_{k=0}^{p-1}C_k^{(2)}\left(\frac{4c^2}{27(c^2+3)}\right)^k\equiv
\begin{cases}
\ds
1 \pmod{p},
&  \text{if $c\in C_0(p)$;} \\
-\frac{9+c}{2c} \pmod{p}, &   \text{if $c\in C_1(p)$;} \\
\frac{9-c}{2c} \pmod{p}, &   \text{if $c\in C_2(p)$.}
\end{cases}
\end{equation*}
\end{theorem}
From Theorem \ref{CaT} and criterion (\ref{cr}) we get the following congruences.
\begin{theorem} \label{MqL2}
Let $q$ be a prime, $q\equiv 1 \pmod{3}$ and so $4q=L^2+27M^2$ with $L, M\in {\mathbb Z}$ and $L\equiv 1\pmod{3}$. Let $p$ be a prime with
$p\ne 2, 3, q$, and let $p\nmid LM$. Then
\begin{align*}
\sum_{k=0}^{p-1}{3k\choose k}\frac{M^{2k}}{q^k}&\equiv
\begin{cases}
\ds
1 \pmod{p},
&  \text{if $p^{\frac{q-1}{3}}\equiv 1\pmod{q}$;} \\[3pt]
\frac{\pm 3M-L}{2L} \pmod{p}, &   \text{if $p^{\frac{q-1}{3}}\equiv \frac{-1\pm L/(3M)}{2}$} \pmod{q},
\end{cases}
\\[5pt]
\sum_{k=0}^{p-1}{3k\choose k}\frac{L^{2k}}{(27q)^k}&\equiv
\begin{cases}
\ds
1 \pmod{p},
&  \text{if $p^{\frac{q-1}{3}}\equiv 1\pmod{q}$;} \\[3pt]
\frac{\pm L-9M}{18M} \pmod{p}, &  \text{if $p^{\frac{q-1}{3}}\equiv \frac{-1\pm 9M/L}{2}$} \pmod{q}
\end{cases}
\end{align*}
and
\begin{align*}
\sum_{k=0}^{p-1}C_k^{(2)}\frac{M^{2k}}{q^k}&\equiv
\begin{cases}
\ds
1 \pmod{p},
&  \text{if $p^{\frac{q-1}{3}}\equiv 1\pmod{q}$;} \\[3pt]
\frac{\pm L-M}{2M} \pmod{p}, &   \text{if $p^{\frac{q-1}{3}}\equiv \frac{-1\pm 9M/L}{2}$} \pmod{q},
\end{cases}
\\[5pt]
\sum_{k=0}^{p-1}C_k^{(2)}\frac{L^{2k}}{(27q)^k}&\equiv
\begin{cases}
\ds
1 \pmod{p},
&  \text{if $p^{\frac{q-1}{3}}\equiv 1\pmod{q}$;} \\[3pt]
\frac{\pm 27M-L}{2L} \pmod{p}, &   \text{if $p^{\frac{q-1}{3}}\equiv \frac{-1\pm L/(3M)}{2}$} \pmod{q}.
\end{cases}
\end{align*}
\end{theorem}
\begin{proof}
Substituting consequently $c=L/(3M)$ and then $c=-9M/L$ in Theorem \ref{CaT} and following the same line of reasoning as in the proof of Theorem \ref{MqL}, we get the above congruences.
\end{proof}
In particular, setting $q=7, 19, 31, 37$ in Theorem \ref{MqL2}, we get the following numerical congruences.
\begin{corollary}
Let $p$ be a prime, $p\ne 2, 3, 7$. Then
\begin{align*}
\sum_{k=0}^{p-1}{3k\choose k}\frac{1}{189^k}&\equiv
\begin{cases}
\ds
-2 \pmod{p},
&  \text{if $p\equiv \pm 2 \pmod{7}$;} \\
\ds 1 \pmod{p}, &  \text{otherwise,}
\end{cases}
\\[7pt]
\sum_{k=0}^{p-1}\frac{C_k^{(2)}}{189^k}&
\equiv
\begin{cases}
\ds
1 \pmod{p},
&  \text{if $p\equiv \pm 1 \pmod{7}$;} \\
\ds -14 \pmod{p}, &   \text{if $p\equiv \pm 2 \pmod{7}$;} \\
\ds 13 \pmod{p}, &   \text{if $p\equiv \pm 3 \pmod{7}$.}
\end{cases}
\end{align*}
\end{corollary}
\begin{corollary}
Let $p$ be a prime, $p\ne 2, 3, 7, 19$. Then
\begin{align*}
\sum_{k=0}^{p-1}{3k\choose k}\frac{1}{19^k}&\equiv
\begin{cases}
\ds
1 \pmod{p},
&  \text{if $p\equiv \pm 1, \pm 7, \pm 8\pmod{19}$;} \\
\ds -2/7 \pmod{p}, &   \text{if $p\equiv \pm 2,\pm 3, \pm 5 \pmod{19}$;} \\
\ds -5/7 \pmod{p}, &   \text{if $ p\equiv \pm 4, \pm 6, \pm 9 \pmod{19}$.}
\end{cases}
\end{align*}
\end{corollary}
\begin{corollary}
Let $p$ be a prime, $p\ne 2, 3, 31$. Then
\begin{align*}
\sum_{k=0}^{p-1}{3k\choose k}\left(\frac{4}{31}\right)^k&\equiv
\begin{cases}
\ds
1 \pmod{p},
&  \text{if $p\equiv \pm 1, \pm 2, \pm 4, \pm 8, \pm 15 \pmod{31}$;} \\
\ds -5/4 \pmod{p}, &   \text{if $p\equiv \pm 3, \pm 6, \pm 7, \pm 12, \pm 14 \pmod{31}$;} \\
\ds 1/4 \pmod{p}, &   \text{if $p\equiv \pm 5,\pm 9, \pm 10, \pm 11, \pm 13 \pmod{31}$.}
\end{cases}
\end{align*}
\end{corollary}
\begin{corollary}
Let $p$ be a prime, $p\ne 2, 3, 11, 37$. Then
\begin{align*}
\sum_{k=0}^{p-1}{3k\choose k}\frac{1}{37^k}&\equiv
\begin{cases}
\ds
1 \pmod{p},
&  \text{if $p\equiv \pm 1, \pm 6, \pm 8, \pm 10, \pm 11, \pm 14, \pmod{37}$;} \\
\ds -4/11 \pmod{p}, &   \text{if $p\equiv \pm 2,\pm 9, \pm 12, \pm 15, \pm 16, \pm 17 \pmod{37}$;} \\
\ds -7/11 \pmod{p}, &   \text{if $p\equiv \pm 3, \pm 4, \pm 5, \pm 7, \pm 13, \pm 18 \pmod{37}$,}
\end{cases}
\\[7pt]
\sum_{k=0}^{p-1}\frac{C_k^{(2)}}{37^k}&
\equiv
\begin{cases}
\ds
1 \pmod{p},
&  \text{if $p\equiv \pm 1, \pm 6, \pm 8, \pm 10, \pm 11, \pm 14, \pmod{37}$;} \\
\ds -6 \pmod{p}, &   \text{if $p\equiv \pm 2,\pm 9, \pm 12, \pm 15, \pm 16, \pm 17 \pmod{37}$;} \\
\ds 5 \pmod{p}, &   \text{if $p\equiv \pm 3, \pm 4, \pm 5, \pm 7, \pm 13, \pm 18 \pmod{37}$.}
\end{cases}
\end{align*}
\end{corollary}

\section{Polynomial congruences involving $S_n$} \label{Section6}

In this section, we will deal with a particular case of Theorem \ref{TM} when $m=6$. In this case, we get polynomial congruences containing the sequence
 $S_k$ (OEIS \seqnum{A176898})  and also $(2k+1)S_k$.

\begin{theorem} \label{T3}
Let $p$ be a prime greater than $3$, and let $t\in D_p$. Then
\begin{align}
\sum_{k=0}^{\lfloor p/6\rfloor}S_k t^k&\equiv \frac{3}{4}\left(\frac{p}{3}\right) w_{\lfloor\frac{p}{6}\rfloor}(1-216t) \pmod{p}, \label{eq18}\\
\sum_{k=(p-1)/2}^{\lfloor 5p/6\rfloor}
S_kt^k&\equiv -\frac{1}{8}\left(\frac{p}{3}\right)\Bigl(w_{\lfloor\frac{5p}{6}\rfloor}(1-216t)+w_{\lfloor\frac{p}{6}\rfloor}(1-216t)\Bigr) \pmod{p}, \nonumber \\
\sum_{k=0}^{\lfloor p/6\rfloor}
(2k+1)S_kt^k&\equiv \frac{1}{2}(-1)^{\frac{p-1}{2}} w_{\lfloor\frac{p}{6}\rfloor}(216t-1) \pmod{p}, \label{eq20}\\
\sum_{k=(p+1)/2}^{\lfloor 5p/6\rfloor}
(2k+1)S_kt^k&\equiv \frac{1}{12}(-1)^{\frac{p-1}{2}} \Bigl(w_{\lfloor\frac{5p}{6}\rfloor}(216t-1)-w_{\lfloor\frac{p}{6}\rfloor}(216t-1)\Bigr) \pmod{p}. \nonumber
\end{align}
\end{theorem}
\begin{corollary} \label{C3}
Let $p$ be a prime greater than $3$, and let $t\in D_p$. Then
\begin{align*}
\sum_{k=0}^{p-1}S_kt^k&\equiv \frac{1}{8}\left(\frac{p}{3}\right)\Bigl(5w_{\lfloor\frac{p}{6}\rfloor}(1-216t)-w_{\lfloor\frac{5p}{6}\rfloor}(1-216t)\Bigr) \pmod{p}, \\
\sum_{k=0}^{p-1}(2k+1)S_k t^k&\equiv \frac{1}{12}(-1)^{\frac{p-1}{2}}\Bigl(w_{\lfloor\frac{5p}{6}\rfloor}(216t-1)+5w_{\lfloor\frac{p}{6}\rfloor}(216t-1)\Bigr) \pmod{p}.
\end{align*}
\end{corollary}

Taking into account (\ref{w-101}), we get the following explicit  congruences. Note that the first congruence below confirms a conjecture of Z.\ W.\ Sun \cite[Conj.\ 2]{ZWS13}.
\begin{corollary}
 Let $p$ be a prime greater than $3$. Then
\begin{align*}
\sum_{k=0}^{p-1}\frac{S_k}{108^k}&\equiv \frac{1}{2}\left(\frac{3}{p}\right),
\qquad\quad
\sum_{k=0}^{\lfloor p/6\rfloor}\frac{S_k}{108^k}\equiv \frac{3}{4}\left(\frac{3}{p}\right) \pmod{p}, \\
\sum_{k=0}^{p-1}\frac{S_k}{216^k}&\equiv \frac{1}{2}\left(\frac{2}{p}\right),
\qquad\quad
\sum_{k=0}^{\lfloor p/6\rfloor}\frac{S_k}{216^k}\equiv \frac{3}{4}\left(\frac{2}{p}\right) \pmod{p},
\end{align*}
\vspace{-0.3cm}
\begin{align*}
\sum_{k=0}^{p-1}\frac{(2k+1)S_k}{108^k}&\equiv \frac{2}{9}\left(\frac{3}{p}\right),
\qquad\quad
\sum_{k=0}^{\lfloor p/6\rfloor}\frac{(2k+1)S_k}{108^k}\equiv \frac{1}{3}\left(\frac{3}{p}\right) \pmod{p},  \\
\sum_{k=0}^{p-1}\frac{(2k+1)S_k}{216^k}&\equiv \frac{1}{2}\left(\frac{6}{p}\right),
\qquad\quad
\sum_{k=0}^{\lfloor p/6\rfloor}\frac{(2k+1)S_k}{216^k}\equiv \frac{1}{2}\left(\frac{6}{p}\right) \pmod{p}.
\end{align*}
\end{corollary}
 From Corollary \ref{C3} and (\ref{12-1})  we get the following congruences.
\begin{corollary}
Let $p$ be a prime greater than $3$. Then
\begin{align*}
\left(\frac{3}{p}\right)\sum_{k=0}^{p-1}\frac{S_k}{432^k}&\equiv
\begin{cases}
\ds
1/2 \pmod{p},
&  \text{if $p\equiv \pm 1 \pmod{9}$;} \\
\ds -11/8 \pmod{p}, &   \text{if $p\equiv \pm 2 \pmod{9}$;} \\
\ds 7/8 \pmod{p}, &   \text{if $p\equiv \pm 4 \pmod{9}$,}
\end{cases}
\\[5pt]
\sum_{k=0}^{p-1}S_k\left(\frac{3}{432}\right)^k&\equiv
\begin{cases}
\ds
1/2 \pmod{p},
&  \text{if $p\equiv \pm 1 \pmod{9}$;} \\
\ds 1/8 \pmod{p}, &  \text{if $p\equiv \pm 2 \pmod{9}$;} \\
\ds-5/8 \pmod{p}, &  \text{if $p\equiv \pm 4 \pmod{9}$,}
\end{cases}
\\[5pt]
\left(\frac{p}{3}\right)\sum_{k=0}^{p-1}\frac{(2k+1)S_k}{432^k}&
\equiv
\begin{cases}
\ds
1/2 \pmod{p},
&  \text{if $p\equiv \pm 1 \pmod{9}$;} \\
\ds -1/12 \pmod{p}, &   \text{if $p\equiv \pm 2 \pmod{9}$;} \\
\ds -5/12 \pmod{p}, &   \text{if $p\equiv \pm 4 \pmod{9}$,}
\end{cases}
\\[5pt]
\sum_{k=0}^{p-1}(2k+1)S_k\left(\frac{3}{432}\right)^k&\equiv
\begin{cases}
\ds
1/2 \pmod{p},
&  \text{if $p\equiv \pm 1 \pmod{9}$;} \\
\ds -3/4 \pmod{p}, &   \text{if $p\equiv \pm 2 \pmod{9}$;} \\
\ds 1/4 \pmod{p}, &   \text{if $p\equiv \pm 4 \pmod{9}$.}
\end{cases}
\end{align*}
\end{corollary}
The following theorem provides two families of polynomial congruences.
\begin{theorem}
Let $p$ be a prime, $p>3$, and let $t\in D_p$.
If $t\not\equiv 0 \pmod{p}$, then the following congruences hold modulo $p$:
\begin{align*}
\sum_{k=0}^{\lfloor p/6\rfloor}S_k\bigl(t^2(4t+1)\bigr)^k&\!\!\equiv\!\!
\frac{1+12t}{32t}\!\left(\frac{1+4t}{p}\right)\!-\!
\frac{1-12t}{32t}\!\left(\frac{1-12t}{p}\right), \\
\sum_{k=0}^{p-1}S_k\!\bigl(t^2(4t+1)\bigr)^k&\!\!\equiv\! \frac{(1+12t)(1+4t)(1-6t)}{32t}
\!\left(\!\frac{1+4t}{p}\!\right)
\!-\!\frac{(1\!-\!12t)(24t^2+6t+1)}{32t}\!\left(\!
\frac{1-12t}{p}\!\right)\!.
\end{align*}
If $6t+1\not\equiv 0 \pmod{p}$, then we have modulo $p$,
\begin{align*}
\sum_{k=0}^{\lfloor p/6\rfloor}(2k+1)S_k\bigl(t^2(4t+1)\bigr)^k\!&\equiv \frac{1+12t}{8(1+6t)}
\left(\frac{1-12t}{p}\right)+\frac{3(1+4t)}{8(1+6t)}
\left(\frac{1+4t}{p}\right), \\
\sum_{k=0}^{p-1}(2k+1)S_k\bigl(t^2(4t+1)\bigr)^k\!&\equiv \frac{(1+12t)(24t^2+6t+1)}{8(1+6t)}\left(\frac{1-12t}{p}\right) \\
& \quad+
\frac{3(1-6t)(1+4t)^2}{8(1+6t)}\left(\frac{1+4t}{p}\right).
\end{align*}
\end{theorem}
\begin{proof}
  From (\ref{eq18}), Corollary \ref{C3} and Lemma \ref{L4.3} for any
$x\in D_p$ with $2x+1\not\equiv 0$ (mod $p$), we have
\begin{align*}
\sum_{k=0}^{\lfloor p/6\rfloor}S_k&\left(\frac{(1-x)(2x+1)^2}{216}\right)^k\!\equiv
\frac{3}{4}\left(\frac{p}{3}\right)w_{\lfloor\frac{p}{6}\rfloor}(4x^3-3x) \\
&\equiv \frac{3(x+1)}{4(2x+1)}
\left(\frac{2x+2}{p}\right)+
\frac{3x}{4(2x+1)}\left(\frac{6-6x}{p}\right)
 \pmod{p},
\end{align*}
\begin{align*}
\sum_{k=0}^{p-1}S_k&\left(\frac{(1-x)(2x+1)^2}{216}\right)^k
\equiv \frac{1}{8}\left(\frac{p}{3}\right)\bigl(5w_{\lfloor\frac{p}{6}\rfloor}(4x^3-3x)
-w_{\lfloor\frac{5p}{6}\rfloor}(4x^3-3x)\bigr) \\
&\equiv \frac{(x+1)(2x^2-x+2)}{4(2x+1)}\left(\frac{2x+2}{p}\right)
-\frac{x(x-1)(2x+3)}{4(2x+1)}\left(\frac{6-6x}{p}\right) \pmod{p}.
\end{align*}
Now replacing $x$ by $(-12t-1)/2$, we get the first two congruences of the theorem.

 Similarly, from (\ref{eq20}), Corollary \ref{C3} and Lemma \ref{L4.3} for any $x\in D_p$ such that $2x+1\not\equiv 0$ (mod $p$),
we have
\begin{align*}
\sum_{k=0}^{\lfloor p/6\rfloor}(2k+1)S_k&\left(\frac{(x+1)(2x-1)^2}{216}\right)^k
\equiv \frac{1}{2}(-1)^{\frac{p-1}{2}}w_{\lfloor\frac{p}{6}\rfloor}(4x^3-3x) \\
&\equiv \frac{x}{2(2x+1)}\left(\frac{2-2x}{p}\right)+\frac{x+1}{2(2x+1)}
\left(\frac{6x+6}{p}\right) \pmod{p},
\end{align*}
\begin{align*}
\sum_{k=0}^{p-1}(2k+1)S_k&\left(\frac{(x+1)(2x-1)^2}{216}\right)^k
\equiv \frac{1}{12}(-1)^{\frac{p-1}{2}}\bigl(5w_{\lfloor\frac{p}{6}\rfloor}(4x^3-3x)+
w_{\lfloor\frac{5p}{6}\rfloor}(4x^3-3x)\bigr) \\ &\equiv
\frac{x(2x^2+x+2)}{6(2x+1)}\left(\frac{2-2x}{p}\right)-
\frac{(x+1)^2(2x-3)}{6(2x+1)}\left(\frac{6x+6}{p}\right) \pmod{p}.
\end{align*}
Replacing $x$ by $(12t+1)/2$, we conclude the proof.
\end{proof}
The next theorem gives a criterion for $c\in C_j(p)$ in terms of values of the sums $\sum_{k=0}^{p-1}S_kt^k$ and $\sum_{k=0}^{p-1}(2k+1)S_kt^k$ modulo $p$.
\begin{theorem} \label{Sc}
Let $p$ be a prime, $p>3$, and let $c\in D_p$ with $c^2\not\equiv -3\pmod{p}$. Then
\begin{equation*}
\left(\frac{3(c^2+3)}{p}\right)\cdot
\sum_{k=0}^{p-1}S_k\left(\frac{c^2}{108(c^2+3)}\right)^k\equiv
\begin{cases}
\ds
1/2 \pmod{p},
&  \text{if $c\in C_0(p)$;} \\
\frac{9-2c}{8c} \pmod{p}, &   \text{if $c\in C_1(p)$;} \\
-\frac{9+2c}{8c} \pmod{p}, &   \text{if $c\in C_2(p)$}
\end{cases}
\end{equation*}
and
\begin{equation*}
\left(\frac{c^2+3}{p}\right)\cdot\sum_{k=0}^{p-1}
\frac{(2k+1)S_k}{36^k(3+c^2)^k}\equiv
\begin{cases}
\ds
1/2 \pmod{p},
&  \text{if $c\in C_0(p)$;} \\
\frac{2-c}{4c} \pmod{p}, &   \text{if $c\in C_1(p)$;} \\
-\frac{2+c}{4c} \pmod{p}, &   \text{if $c\in C_2(p)$.}
\end{cases}
\end{equation*}
\end{theorem}
\begin{proof}
  From Corollary \ref{C3} we have
\begin{align}
\sum_{k=0}^{p-1}S_k\left(\frac{c^2}{108(3+c^2)}\right)^k&\equiv
\frac{1}{8}\left(\frac{p}{3} \right)\left(5w_{\lfloor\frac{p}{6}\rfloor}\left(\frac{3-c^2}{3+c^2}\right)
-w_{\lfloor\frac{5p}{6}\rfloor}\left(\frac{3-c^2}{3+c^2}\right)\right) \pmod{p},
\label{S0} \\
\sum_{k=0}^{p-1}\frac{(2k+1)S_k}{36^k(3+c^2)^k}&\equiv
\frac{(-1)^{\frac{p-1}{2}}}{12} \left(5w_{\lfloor\frac{p}{6}\rfloor}\left(\frac{3-c^2}{3+c^2}\right)
+w_{\lfloor\frac{5p}{6}\rfloor}\left(\frac{3-c^2}{3+c^2}\right)\right)
\pmod{p} \label{S1}.
\end{align}
If $p\equiv 1$ (mod $6$), then $p$ splits into the product of primes in
${\mathbb Z}[\omega]$, $p=\pi\overline{\pi}$ with $\pi\equiv 2$ (mod~$3$) and, by (\ref{tr2}), we easily find
$$
w_{\lfloor\frac{p}{6}\rfloor}\left(\frac{3-c^2}{3+c^2}\right)=
\frac{(-1)^{\frac{p-1}{6}}}{2c(c^2+3)^{\frac{p-1}{6}}}\left(
(c-1-2\omega)^{\frac{p-1}{3}+1}+(c+1+2\omega)^{\frac{p-1}{3}+1}\right).
$$
Applying Lemma \ref{LR}, we have
$$
w_{\lfloor\frac{p}{6}\rfloor}\left(\frac{3-c^2}{3+c^2}\right)\!\equiv\!
\frac{(-1)^{\frac{p-1}{6}}}{2c(c^2+3)^{\frac{p-1}{2}}}\left(\!
(c-1-2\omega)\left(
\!\frac{c+1+2\omega}{p}\!\right)_3\!\!\!+
(c+1+2\omega)
\left(\!\frac{c+1+2\omega}{p}\!\right)^{-1}_3\right)
$$
modulo $\pi$ and therefore,
\begin{equation} \label{p16}
(-1)^{\frac{p-1}{2}}\left(\frac{c^2+3}{p}\right)\cdot
w_{\lfloor\frac{p}{6}\rfloor}\left(\frac{3-c^2}{3+c^2}\right)\equiv
\begin{cases}
\ds
1 \pmod{p},
&  \text{if $c\in C_0(p)$;} \\
\frac{3-c}{2c} \pmod{p}, &   \text{if $c\in C_1(p)$;} \\
-\frac{3+c}{2c} \pmod{p}, &   \text{if $c\in C_2(p)$.}
\end{cases}
\end{equation}
If $p\equiv 5$ (mod $6$), then
$$
w_{\lfloor\frac{p}{6}\rfloor}\left(\frac{3-c^2}{3+c^2}\right)=
\frac{(-1)^{\frac{p-5}{6}}}{2c(c^2+3)^{\frac{p-5}{6}}}\left(
(c-1-2\omega)^{\frac{p+1}{3}-1}+(c+1+2\omega)^{\frac{p+1}{3}-1}\right).
$$
Now, by Lemma \ref{LR}, we get
$$
w_{\lfloor\frac{p}{6}\rfloor}\left(\frac{3-c^2}{3+c^2}\right)\!\equiv\!
\frac{(-1)^{\frac{p-5}{6}}}{2c(c^2+3)^{\frac{p-1}{2}}}\left(\!
(c+1+2\omega)\left(\!
\frac{c+1+2\omega}{p}\!\right)_3^{-1}\!\!\!+
(c-1-2\omega)
\left(\!\frac{c+1+2\omega}{p}\!\right)_3\right)
$$
modulo  $p$ and therefore,
\begin{equation} \label{p56}
(-1)^{\frac{p-5}{6}}\left(\frac{c^2+3}{p}\right)\cdot
w_{\lfloor\frac{p}{6}\rfloor}\left(\frac{3-c^2}{3+c^2}\right)\equiv
\begin{cases}
\ds
1 \pmod{p},
&  \text{if $c\in C_0(p)$;} \\
\frac{3-c}{2c} \pmod{p}, &   \text{if $c\in C_1(p)$;} \\
-\frac{3+c}{2c} \pmod{p}, &   \text{if $c\in C_2(p)$.}
\end{cases}
\end{equation}
Comparing (\ref{p16}) and (\ref{p56}), we get that (\ref{p16}) holds  for all primes $p>3$.

Applying the similar argument for evaluation of $w_{\lfloor\frac{5p}{6}\rfloor}\bigl(\frac{3-c^2}{3+c^2}\bigr)$, we see that if $p\equiv 1$ (mod~$6$), then $5p\equiv 5$ (mod $6$) and we have
$$
w_{\lfloor\frac{5p}{6}\rfloor}\left(\frac{3-c^2}{3+c^2}\right)=
\frac{(-1)^{\frac{p-1}{6}}}{2c(c^2+3)^{\frac{5(p-1)}{6}}}\left(
(c-1-2\omega)^{\frac{5(p-1)}{3}+1}+(c+1+2\omega)^{\frac{5(p-1)}{3}+1}\right).
$$
Now, by Lemma \ref{LR}, we easily find
$$
w_{\lfloor\frac{5p}{6}\rfloor}\left(\frac{3\!-\!c^2}{3\!+\!c^2}\right)\!\equiv\!
\frac{(-1)^{\frac{p-1}{6}}}{2c(c^2+3)^{\frac{5(p-1)}{2}}}\!\left(\!
(c-1-2\omega)\!\left(\!
\frac{c+1+2\omega}{p}\!\right)^{\!5}_3
\!\!\!+\!
(c+1+2\omega)\!
\left(\!\frac{c+1+2\omega}{p}\!\right)^{\!-5}_3\right)
$$
modulo $\pi$, which implies
\begin{equation} \label{5p1}
(-1)^{\frac{p-1}{2}}\left(\frac{c^2+3}{p}\right)\cdot
w_{\lfloor\frac{5p}{6}\rfloor}\left(\frac{3-c^2}{3+c^2}\right)\equiv
\begin{cases}
\ds
1 \pmod{p},
&  \text{if $c\in C_0(p)$;} \\
-\frac{3+c}{2c} \pmod{p}, &  \text{if $c\in C_1(p)$;} \\
\frac{3-c}{2c} \pmod{p}, &   \text{if $c\in C_2(p)$.}
\end{cases}
\end{equation}
Similarly, if $p\equiv 5$ (mod $6$), then $5p\equiv 1$ (mod $6$) and we have
$$
w_{\lfloor\frac{5p}{6}\rfloor}\left(\frac{3-c^2}{3+c^2}\right)=
\frac{(-1)^{\frac{5p-1}{6}}}{2c(c^2+3)^{\frac{5p-1}{6}}}\left(
(c-1-2\omega)^{\frac{5p+2}{3}}+(c+1+2\omega)^{\frac{5p+2}{3}}\right).
$$
By Lemma \ref{LR}, we readily get
$$
w_{\lfloor\frac{5p}{6}\rfloor}\left(\frac{3\!-\!c^2}{3\!+\!c^2}\right)\equiv
\frac{(-1)^{\frac{5p-1}{6}}}{2c(c^2+3)^{\frac{5p-1}{2}}}\left(\!
(c+1+2\omega)\left(\!
\frac{c+1+2\omega}{p}\!\right)_3^{\!-5}\!\!\!\!\!+\!
(c-1-2\omega)
\left(\!\frac{c+1+2\omega}{p}\!\right)^{\!5}_3\right)
$$
modulo $p$
and therefore after simplification we obtain that (\ref{5p1}) holds for all primes $p>3$.
Finally, substituting (\ref{p16}) and (\ref{5p1}) into (\ref{S0}) and
(\ref{S1}), we get the congruences of the theorem.
\end{proof}
  From Theorem \ref{Sc} with $c=-9M/L$  and $c=L/3M$ and criterion (\ref{cr}) we get the following congruences.
\begin{theorem} \label{MqLS}
Let $q$ be a prime, $q\equiv 1 \pmod{3}$ and so $4q=L^2+27M^2$ with $L, M\in {\mathbb Z}$ and $L\equiv 1\pmod{3}$. Let $p$ be a prime with
$p\ne 2, 3, q$, and let $p\nmid LM$. Then
\begin{align*}
\left(\frac{q}{p}\right)\sum_{k=0}^{p-1}S_k\frac{M^{2k}}{(16q)^k}&\equiv
\begin{cases}
\ds
1/2 \pmod{p},
&  \text{if $p^{\frac{q-1}{3}}\equiv 1\pmod{q}$;} \\[3pt]
\frac{\pm L-2M}{8M} \pmod{p}, &  \text{if $p^{\frac{q-1}{3}}\equiv \frac{-1\pm L/(3M)}{2}$} \pmod{q},
\end{cases}
\\[7pt]
\left(\frac{3q}{p}\right)\sum_{k=0}^{p-1}S_k\frac{L^{2k}}{(432q)^k}&\equiv
\begin{cases}
\ds
1/2 \pmod{p},
&  \text{if $p^{\frac{q-1}{3}}\equiv 1\pmod{q}$;} \\[3pt]
\frac{\pm 27M-2L}{8L} \pmod{p}, &   \text{if $p^{\frac{q-1}{3}}\equiv \frac{-1\pm 9M/L}{2}$} \pmod{q}
\end{cases}
\end{align*}
and
\begin{align*}
\left(\frac{q}{p}\right)\sum_{k=0}^{p-1}(2k+1)S_k\frac{M^{2k}}{(16q)^k}&\equiv
\begin{cases}
\ds
1/2 \pmod{p},
&  \text{if $p^{\frac{q-1}{3}}\equiv 1\pmod{q}$;} \\[3pt]
\frac{\pm 6M-L}{4L} \pmod{p}, &   \text{if $p^{\frac{q-1}{3}}\equiv \frac{-1\pm 9M/L}{2}$} \pmod{q},
\end{cases}
\\[7pt]
\left(\frac{3q}{p}\right)\sum_{k=0}^{p-1}(2k+1)S_k\frac{L^{2k}}{(432q)^k}&\equiv
\begin{cases}
\ds
1/2 \pmod{p},
&  \text{if $p^{\frac{q-1}{3}}\equiv 1\pmod{q}$;} \\[3pt]
\frac{\pm 2L-9M}{36M} \pmod{p}, &   \text{if $p^{\frac{q-1}{3}}\equiv \frac{-1\pm L/(3M)}{2}$} \pmod{q}.
\end{cases}
\end{align*}
\end{theorem}
For example, if $q=7$, then $4q=L^2+27M^2$ with $L=M=1$ and by Theorem~\ref{MqLS}, we get the following numerical congruences.
\begin{corollary}
Let $p$ be a prime, $p\ne 2, 3, 7$. Then
\begin{align*}
\left(\frac{7}{p}\right)\sum_{k=0}^{p-1}\frac{S_k}{112^k}&\equiv
\begin{cases}
\ds
1/2 \pmod{p},
&  \text{if $p\equiv \pm 1\pmod{7}$;} \\
\ds-3/8 \pmod{p}, &   \text{if $p\equiv \pm 2 \pmod{7}$;} \\
\ds-1/8 \pmod{p}, &   \text{if $p\equiv \pm 3 \pmod{7}$,}
\end{cases}
\\[7pt]
\left(\frac{21}{p}\right)\sum_{k=0}^{p-1}\frac{S_k}{3024^k}&\equiv
\begin{cases}
\ds
1/2 \pmod{p},
&  \text{if $p\equiv \pm 1\pmod{7}$;} \\
\ds 25/8 \pmod{p}, &   \text{if $p\equiv \pm 2 \pmod{7}$;} \\
\ds-29/8 \pmod{p}, &  \text{if $p\equiv \pm 3 \pmod{7}$,}
\end{cases}
\\[7pt]
\left(\frac{7}{p}\right)\sum_{k=0}^{p-1}\frac{(2k+1)S_k}{112^k}&
\equiv  \begin{cases}
\ds
1/2 \pmod{p},
&  \text{if $p\equiv \pm 1\pmod{7}$;} \\
\ds 5/4 \pmod{p}, &   \text{if $p\equiv \pm 2 \pmod{7}$;} \\
\ds-7/4 \pmod{p}, &   \text{if $p\equiv \pm 3 \pmod{7}$,}
\end{cases}
\\[7pt]
\left(\frac{21}{p}\right)\sum_{k=0}^{p-1}\frac{(2k+1)S_k}{3024^k}&
\equiv
\begin{cases}
\ds
1/2 \pmod{p},
&  \text{if $p\equiv \pm 1\pmod{7}$;} \\
\ds -11/36 \pmod{p}, &   \text{if $p\equiv \pm 2 \pmod{7}$;} \\
\ds-7/36 \pmod{p}, &   \text{if $p\equiv \pm 3 \pmod{7}$.}
\end{cases}
\end{align*}
\end{corollary}
Similarly, setting $q=13, 19, 31$ in Theorem \ref{MqLS}, we obtain the following congruences.
\begin{corollary}
Let $p$ be a prime, $p\ne 2, 3, 5, 13$. Then
\begin{align*}
\left(\frac{13}{p}\right)\sum_{k=0}^{p-1}\frac{S_k}{208^k}&\equiv
\begin{cases}
\ds
1/2 \pmod{p},
&  \text{if $p\equiv \pm 1, \pm 5\pmod{13}$;} \\
\ds-7/8 \pmod{p}, &   \text{if $p\equiv \pm 2,\pm 3 \pmod{13}$;} \\
\ds 3/8 \pmod{p}, &   \text{if $p\equiv \pm 4, \pm 6 \pmod{13}$,}
\end{cases}
\\[7pt]
\left(\frac{39}{p}\right)\sum_{k=0}^{p-1}S_k\left(\frac{25}{5616}\right)^k&\equiv
\begin{cases}
\ds
1/2 \pmod{p},
&  \text{if $p\equiv \pm 1, \pm 5\pmod{13}$;} \\
\ds 17/40 \pmod{p}, &   \text{if $p\equiv \pm 2, \pm 3 \pmod{13}$;} \\
\ds-37/40 \pmod{p}, &  \text{if $p\equiv \pm 4, \pm 6 \pmod{13
}$,}
\end{cases}
\\[7pt]
\left(\frac{13}{p}\right)\sum_{k=0}^{p-1}\frac{(2k+1)S_k}{208^k}&
\equiv
\begin{cases}
\ds
1/2 \pmod{p},
& \text{if $p\equiv \pm 1, \pm 5\pmod{13}$;} \\
\ds 1/20 \pmod{p}, &  \text{if $p\equiv \pm 2,\pm 3 \pmod{13}$;} \\
\ds -11/20 \pmod{p}, &  \text{if $p\equiv \pm 4, \pm 6 \pmod{13}$,}
\end{cases}
\\[7pt]
\left(\frac{39}{p}\right)\sum_{k=0}^{p-1}(2k+1)S_k\left(\frac{25}{5616}\right)^k&\equiv
\begin{cases}
\ds
1/2 \pmod{p},
&  \text{if $p\equiv \pm 1, \pm 5\pmod{13}$;} \\
\ds -19/36 \pmod{p}, &   \text{if $p\equiv \pm 2, \pm 3 \pmod{13}$;} \\
\ds 1/36 \pmod{p}, &   \text{if $p\equiv \pm 4, \pm 6 \pmod{13
}$.}
\end{cases}
\end{align*}
\end{corollary}
\begin{corollary}
Let $p$ be a prime, $p\ne 2, 3, 7, 19$. Then
\begin{align*}
\left(\frac{19}{p}\right)\sum_{k=0}^{p-1}\frac{S_k}{304^k}&\equiv
\begin{cases}
\ds
1/2 \pmod{p},
&  \text{if $p\equiv \pm 1, \pm 7, \pm 8\pmod{19}$;} \\
\ds 5/8 \pmod{p}, &  \text{if $p\equiv \pm 2,\pm 3, \pm 5 \pmod{19}$;} \\
\ds -9/8 \pmod{p}, &   \text{if $p\equiv \pm 4, \pm 6, \pm 9 \pmod{19}$,}
\end{cases}
\\[7pt]
\left(\frac{19}{p}\right)\sum_{k=0}^{p-1}\frac{(2k+1)S_k}{304^k}&
\equiv
\begin{cases}
\ds
1/2 \pmod{p},
&  \text{if $p\equiv \pm 1, \pm 7, \pm 8\pmod{19}$;} \\
\ds -13/28 \pmod{p}, &  \text{if $p\equiv \pm 2,\pm 3, \pm 5 \pmod{19}$;} \\
\ds -1/28 \pmod{p}, &   \text{if $p\equiv \pm 4, \pm 6, \pm 9 \pmod{19}$.}
\end{cases}
\end{align*}
\end{corollary}
\begin{corollary}
Let $p$ be a prime, $p\ne 2, 3, 31$. Then
\begin{align*}
\left(\frac{31}{p}\right)\sum_{k=0}^{p-1}\frac{S_k}{124^k}&\equiv
\begin{cases}
\ds
1/2 \pmod{p},
&  \text{if $p\equiv \pm 1, \pm 2, \pm 4, \pm 8, \pm 15\pmod{31}$;} \\
\ds -1/2 \pmod{p}, &   \text{if $p\equiv \pm 3, \pm 6, \pm 7, \pm 12, \pm 14 \pmod{31}$;} \\
\ds 0 \pmod{p}, &   \text{if $p\equiv \pm 5, \pm 9, \pm 10, \pm 11, \pm 13 \pmod{31}$,}
\end{cases}
\\[7pt]
\left(\frac{93}{p}\right)\sum_{k=0}^{p-1}\frac{S_k}{837^k}&\equiv
\begin{cases}
\ds
1/2 \pmod{p},
&  \text{if $p\equiv \pm 1, \pm 2, \pm 4, \pm 8, \pm 15\pmod{31}$;} \\
\ds 23/16 \pmod{p}, &   \text{if $p\equiv \pm 3, \pm 6, \pm 7, \pm 12, \pm 14 \pmod{31}$;} \\
\ds -31/16 \pmod{p}, &  \text{if $p\equiv \pm 5, \pm 9, \pm 10, \pm 11, \pm 13 \pmod{31}$,}
\end{cases}
\end{align*}
\begin{align*}
\left(\frac{31}{p}\right)\sum_{k=0}^{p-1}\frac{(2k+1)S_k}{124^k}&\equiv
\begin{cases}
\ds
-1 \pmod{p},
&  \text{if $p\equiv \pm 5,  \pm 9,  \pm 10, \pm 11, \pm 13
\pmod{31}$;} \\
\ds 1/2 \pmod{p}, &   \text{otherwise,}
\end{cases}
\\[7pt]
\left(\frac{93}{p}\right)\sum_{k=0}^{p-1}\frac{(2k+1)S_k}{837^k}&
\equiv
\begin{cases}
\ds
1/2 \pmod{p},
&  \text{if $p\equiv \pm 1, \pm 2, \pm 4, \pm 8, \pm 15\pmod{31}$;} \\
\ds -13/36 \pmod{p}, &   \text{if $p\equiv \pm 3, \pm 6, \pm 7, \pm 12, \pm 14 \pmod{31}$;} \\
\ds -5/36 \pmod{p}, &   \text{if $p\equiv \pm 5, \pm 9, \pm 10, \pm 11, \pm 13 \pmod{31}$.}
\end{cases}
\end{align*}
\end{corollary}

\section{Closed form for a companion sequence of $S_n$} \label{Section7}

As we noticed in the Introduction, the sequence $S_n$ can be defined explicitly by formula (\ref{Sn}) or by the generating function (\ref{gfSn}). Sun \cite{ZWS13} considered a companion sequence  $T_n$, whose definition comes from a conjectural series expansion of trigonometric functions \cite[Conj.\ 4]{ZWS13}: there are positive integers $T_1$, $T_2$, $T_3, \ldots$ such that
\begin{equation} \label{conj}
\sum_{k=0}^{\infty}S_kx^{2k+1}+\frac{1}{24}-\sum_{k=1}^{\infty}T_kx^{2k}=\frac{1}{12}\,\cos\left(\frac{2}{3}\arccos(6\sqrt{3}x)\right)
\end{equation}
for all real $x$ with $|x|\le 1/(6\sqrt{3})$.
The first few values of $T_n$ are as follows:
$$
1, \,\,\, 32, \,\,\, 1792, \,\,\, 122880, \,\,\, 9371648, \,\,\, 763363328,
\ldots.
$$
In this section, we give an exact formula for $T_n$. It easily follows from the companion series expansion  to (\ref{eq04}) \cite[p.\ 210,(12)]{Lu}:
\begin{equation} \label{cosa}
\cos(a\arcsin(z))=F\left(-\frac{a}{2},\, \frac{a}{2};\frac{1}{2};z^2 \right), \qquad |z|\le 1.
\end{equation}
\begin{proposition}
The coefficients $T_k$, $k\ge 1$, in expansion {\rm(\ref{conj})} are given by
$$
T_k=\frac{16^{k-1}}{k}{3k-2\choose 2k-1}=16^{k-1}\left(
2{3k-2\choose k-1}-{3k-2\choose k}\right).
$$
\end{proposition}
\begin{proof}
Combining formulas (\ref{eq04}) and (\ref{cosa}) with  the obvious trigonometric identity
$$
\arcsin(z)+\arccos(z)=\frac{\pi}{2},
$$
we get a transformation formula connecting both  hypergeometric functions from (\ref{eq04}) and~(\ref{cosa}):
$$
\cos\left(\frac{\pi a}{2}\right)\!F\!\left(-\frac{a}{2}, \frac{a}{2}; \frac{1}{2}; z^2\right)
+\sin\left(\frac{\pi a}{2}\right)azF\!\left(\frac{1+a}{2},\frac{1-a}{2}; \frac{3}{2}; z^2\right)=\cos(a\arccos(z)), \,\,\, |z|\le 1.
$$
Plugging in $a=2/3$, we get
$$
\frac{1}{2}F\left(-\frac{1}{3}, \frac{1}{3}; \frac{1}{2}; z^2\right)
+\frac{z}{\sqrt{3}}F\left(\frac{1}{6}, \frac{5}{6}; \frac{3}{2}; z^2\right)=\cos\left(\frac{2}{3}\arccos(z)\right), \quad |z|\le 1.
$$
Replacing $z$ by $6\sqrt{3}x$ with $|x|\le 1/(6\sqrt{3})$ and
taking into account that
$$
F\left(\frac{1}{6}, \frac{5}{6}; \frac{3}{2}; 108x^2\right)=
2\sum_{k=0}^{\infty}S_kx^{2k},
$$
we obtain
$$
\frac{1}{24}F\left(-\frac{1}{3}, \frac{1}{3}; \frac{1}{2}; 108x^2\right)+\sum\limits_{k=0}^{\infty}S_kx^{2k+1}=
\frac{1}{12}\,\cos\left(\frac{2}{3}\arccos(6\sqrt{3}x)\right),
$$
which gives the following generating function for the companion sequence $T_n:$
$$
\frac{1}{24}-\sum_{k=1}^{\infty}T_kx^{2k}=\frac{1}{24}
F\left(-\frac{1}{3}, \frac{1}{3}; \frac{1}{2}; 108x^2\right).
$$
Comparing coefficients of powers of $x^2$, we get a formula for $T_k$,
$$
T_k=-\frac{1}{24}\frac{(-1/3)_k(1/3)_k}{(1/2)_k k!}\,108^k=
\frac{16^{k-1}}{k}{3k-2\choose 2k-1}=
16^{k-1}\left(2{3k-2\choose k-1}-{3k-2\choose k}\right),
$$
which shows that $T_k\in {\mathbb N}$ for all positive integers $k$.
\end{proof}

\section{Acknowledgement}

The authors would like to thank the editor-in-chief  and the referees for their time and useful comments.
 The authors gratefully acknowledge support from the  Research Immersion Fellowships of the Fields Institute.

%\bigskip
%\hrule
%\bigskip

%\noindent 2010 {\it Mathematics Subject Classification}:
%Primary 11A07; Secondary
%11A15, 11B37, 11B65, 33C05, 33C45.

%\noindent \emph{Keywords: }
%Jacobi polynomial, higher-order Catalan number, polynomial congruence, cubic residue, generating function.

\bigskip
\hrule
\bigskip

\noindent (Concerned with sequences  \seqnum{A000108}, \seqnum{A001448}, \seqnum{A001764}, \seqnum{A005809}, \seqnum{A048990}, \seqnum{A176898}
)

\bigskip
\hrule
\bigskip

\end{document}